\documentclass[12pt]{amsart}
\usepackage{amssymb, amsmath, amsfonts}

\def\phi{\varphi}
\def\eps{\varepsilon}
\def\d{\partial}
\def\a{\alpha}
\def\b{\beta}
\def\g{\gamma}

\def\l{\lambda}

\def\Z{{\mathbb Z}}

\def\e{\epsilon}

\def\sp{{\rm span\,}}
\def\Aut{{\sf Aut\,}}
\def\Hom{{\sf Hom\,}}

\def\1#1{\overline{#1}}
\def\2#1{\widetilde{#1}}
\def\3#1{\widehat{#1}}
\def\4#1{\mathbb{#1}}
\def\5#1{\mathfrak{#1}}
\def\6#1{{\mathcal{#1}}}

\def\C{{\4C}}

\def\Z{{\4Z}}
\def\T{{\Theta}}
\newtheorem{theorem}{Theorem}[section]

\newtheorem{corollary}[theorem]{Corollary}

\newtheorem{definition}[theorem]{Definition}
\newtheorem{lemma}[theorem]{Lemma}

\newtheorem{proposition}[theorem]{Proposition}

\theoremstyle{remark}
\newtheorem{example}[theorem]{Example}

\begin{document}
\numberwithin{equation}{section}

\def\bl{\begin{Lem}}
\def\el{\end{Lem}}
\def\bp{\begin{Pro}}
\def\ep{\end{Pro}}
\def\bt{\begin{Thm}}
\def\et{\end{Thm}}
\def\bc{\begin{Cor}}
\def\ec{\end{Cor}}
\def\bd{\begin{Def}}
\def\ed{\end{Def}}
\def\br{\begin{Rem}}
\def\er{\end{Rem}}
\def\be{\begin{example}}
\def\ee{\end{example}}
\def\bpf{\begin{proof}}
\def\epf{\end{proof}}
\def\ben{\begin{enumerate}}
\def\een{\end{enumerate}}
\def\beq{\begin{equation}}
\def\eeq{\end{equation}}

\def\Label#1{\label{#1}}

\title[Rigidity of CR maps ]{Rigidity of proper holomorphic maps between bounded symmetric
domains}

\author[S.-Y. Kim \& D. Zaitsev]{Sung-Yeon Kim* and Dmitri Zaitsev**}
\address{S.-Y. Kim: Department of Mathematics Education, Kangwon National University, 123 Hyoja-dong, Chuncheon, Kangwon-do, 200-701, Korea }
\email{sykim87@kangwon.ac.kr}
\address{D. Zaitsev: School of Mathematics, Trinity College Dublin, Dublin 2, Ireland}
\email{zaitsev@maths.tcd.ie}

%\thanks{2000 {\it Mathematics Subject Classification}. primary;32V40, secondary;32M05}
\thanks{*This research was supported by Basic Science Research Program through the National Research Foundation of Korea (NRF) funded by the Ministry of Education, Science and Technology (grant number 2009-0067947)}
\thanks{**Supported in part by the Science Foundation Ireland grant 10/RFP/MTH2878.}

%\dedicatory{}
%\thanks{This paper is in final form and no version of it will be submitted
%for publication elsewhere.}

\begin{abstract}
Our first main result gives assumptions guaranteeing that proper holomorphic maps between Cartan type I bounded symmetric domains have simple block matrix shape, answering positively a question of Mok.
The proof is based on the second main result establishing similar phenomenon for local CR maps between
arbitrary boundary components of two bounded symmetric domains of the above type. Since boundary components other than Shilov boundaries are Levi-degenerate, our analysis is based on their $2$-nondegeneracy
combining Levi forms with higher order tensors.
\end{abstract}

\maketitle

\section{Introduction}
The goal of this paper is to prove new rigidity results
for proper holomorphic maps between 
bounded symmetric domains.
In fact, we obtain our results for 
maps only defined locally near a boundary point
and sending open pieces of boundaries into each other.
Furthermore, we also provide a pure CR version
of our result for CR maps between boundary components
of bounded symmetric domains.
%In particular, we answer a question of Mok in our situation.

Since the work of Bochner \cite{Bo47} and Calabi \cite{Ca53},
rigidity properties of holomorphic isometries between bounded symmetric domains
attracted considerable attention.
The reader is referred to the survey by Mok \cite{M11} for extensive discussion.
See also the work of  Siu \cite{S80, S81}
for other important rigidity phenomena for bounded symmetric domains,
such as the strong rigidity of complex structures of their compact quotients.

Remarkably, many rigidity properties survive when the isometry condition 
is replaced by purely topological conditions such as {\em properness}.
(Recall that a map between topological spaces is called {\em proper}
if its preimages of compact subsets are compact.)
The work on rigidity of {\em proper holomorphic maps} goes back to 
the work of Poincar\'e \cite{P07} and later Alexander \cite{A74}
for maps between balls of equal dimension, or more generally, 
one-sided neighborhoods of their boundary points.
However, by intriguing contrast, proper holomorphic maps between balls of 
{\em different dimensions} lack similar rigidity properties,
see the work of  Hakim-Sibony \cite{HS}, L\o w \cite{Lo}, Forstneri\v c \cite{F86a}, Globevnik~\cite{G}, Stens\o nes~\cite{St}.
On the other hand, rigidity can be regained by strengthening properness
by requiring additional boundary regularity, see the work of Webster \cite{W79}, 
Faran \cite{Fa86}, Cima-Suffridge \cite{CS83,CS90}, Forstneri\v c \cite{F86b, F89},
Huang \cite{H99,H03}, Huang-Ji \cite{HJ01}, Huang-Ji-Xu \cite{HJX06}, Ebenfelt-Minor \cite{EM12} and Ebenfelt~\cite{E13}.
In another direction, further rigidity phonemena for CR
maps between real hypersurfaces and  {\em hyperquadrics} have been discovered by
Ebenfelt, Huang and the second author \cite{EHZ04,EHZ05},
Baouendi-Huang \cite{BH05}, Baouendi-Ebenfelt-Huang \cite{BEH08,BEH09},
Ebenfelt-Shroff \cite{ES10} and Ng \cite{Ng13c}.

In contrast to holomorphic maps between balls (or CR maps between hypersurfaces),
rigidity properties for maps between 
{\em bounded symmetric domains $D$ and $D'$ of higher rank}
are much less understood.
If the rank $r'$ of $D'$ does not exceed the rank $r$ of $D$
and both ranks $r,r'\ge2$, 
the rigidity of {\em proper holomorphic maps} $f\colon D\to D'$
was conjectured by Mok \cite{M89} and proved by Tsai \cite{T93},
showing that $f$ is necessarily totally geodesic (with respect to the Bergmann metric).
In the remaining case $r<r'$, very little seems to be known,
see the work of Tu \cite{Tu02a,Tu02b}, Mok \cite{M08}
and more recently Mok-Ng-Tu \cite{MNT10},  Mok-Ng \cite{MN12},  Ng \cite{Ng12, Ng13a, Ng13b}.
%See also Mok-Ng-Tu \cite{MNT10}

In \cite{KZ}, the authors established rigidity for local CR embeddings 
between Shilov boundaries of Cartan type I bounded symmetric domains $D_{p,q}$ and $D_{p',q'}$ of any rank under the assumption 
\begin{equation}\Label{pq-assumption}
p'-q'<2(p-q)
\end{equation}
(corresponding to the known assumption $n'<2n$ for 
maps between balls in $\C^{n+1}$ and $\C^{n'+1}$ or their boundaries.)
Recall that the Cartan type I bounded symmetric domain $D_{p,q}$ is the set of $p\times q$ matrices $z$ over $\C$ such that 
$I_q-z^*z$ is positive definite,
where $I_q$ is the identity $q\times q$ matrix and $z^*=\bar z^t$.
In \cite{KZ}, examples were also given of maps of ``Whitney type'' showing that \eqref{pq-assumption} cannot be dropped.
However, even though these examples are (polynomial) CR maps between Shilov boundaries,
and map $D_{p,q}$ into $D_{p',q'}$, they in general do not induce {\em proper} maps between  these domains,
unless the rank $r=q=1$. 
Nevertheless, also for proper holomorphic maps between bounded symmetric domains of Cartan type I,
rigidity is known to fail (see e.g.\ \cite{T93}) due to the presence of maps of the block matrix form
\begin{equation}\Label{split-form}
f\colon D_{p,q}\to D_{p',q'}, \quad
z\mapsto
\begin{pmatrix}
 z & 0 \\
 0 & h(z)
\end{pmatrix},
\end{equation}
where $h(z)$ is arbitrary holomorphic matrix-valued function satisfying
\begin{equation}\Label{h-ineq}
I_{q'-q}-h(z)^* h(z) \text{ is positive definite}, \quad z\in D_{p,q}.
\end{equation}
In view of this fact, N. Mok asked the following question:

\bigskip

``{\em Are proper holomorphic maps between bounded symmetric domains of higher rank, after composing with suitable automorphisms of the domains, always of the form \eqref{split-form}}?''

\bigskip

In this paper we consider a situation where we can answer this question affirmatively. 
In fact, we replace proper maps by more general locally defined ones but have to assume
some boundary regularity. Recall (see e.g.\ \cite{KZ03}) that the boundary $\d D_{p,q}$ is a union of $q$ smooth submanifolds (boundary components). We call $x\in\d D_{p,q}$ is a {\em smooth boundary point},
if $\d D_{p,q}$ is a smooth hypersurface in a neighborhood of $x$.
As our first main result we prove:

\begin{theorem}\Label{cor0}
Let $U\subset\C^{p\times q}$ $(p\geq q>1)$ be an open neighborhood of a
smooth boundary point $x\in \d D_{p,q}$ 
and 
 $f\colon U\cap \1{D_{p,q}} \to \1{D_{p',q'}}$ be a smooth map,
 holomorphic in $U\cap D_{p,q}$
 with $f(U\cap \d D_{p,q})\subset \d D_{p',q'}$
 but $f(U\cap D_{p,q})\not\subset \d D_{p',q'}$.
Assume that
\begin{equation}\Label{cor-ineq}
p'<2p-1,\quad q'<p.
\end{equation}
Then $p'\geq p$, $q'\geq q$ and after composing with suitable automorphisms of $D_{p,q}$ and $D_{p',q'}$,
$f$ takes the block matrix form \eqref{split-form}
with $h$ satisfying \eqref{h-ineq}.
\end{theorem}

Note that the case $q=q'=1$ corresponds to both $D_{p,q}$ and $D_{p',q'}$ being unit balls,
where the same conclusion (also under weaker regularity) is due to Huang~\cite{H99}.
In this case, the first inequality \eqref{cor-ineq} is sharp whereas the second is automatically satisfied.
As immediate application of Theorem~\ref{cor0} for proper holomorphic maps, we obtain:

\begin{corollary}\Label{cor}
Let $f\colon D_{p,q} \to D_{p',q'}$ $(p\ge q>1)$ be a proper holomorphic map which extends smoothly to a neighborhood of a smooth boundary point.
Then assuming \eqref{cor-ineq} we obtain the conclusion of 
Theorem~{\rm\ref{cor0}}.
\end{corollary}

The main difference from the situation of \cite{KZ} here
is that a proper holomorphic map, even if smoothly extendible to the boundary
(and hence sending boundaries into each other),
need not send Shilov boundaries into each other,
unless the source domain is of rank $1$ (i.e.\ the ball).
In higher rank case considered here,
boundary extensions of proper holomorphic maps 
will send boundary components of the source domain
into some of those of the target.
Thus in order to establish Theorem~\ref{cor0},
we need to analyze CR maps between general boundary components of $D_{p,q}$ and $D_{p',q'}$.
For $p\ge q\geq r\geq 1$, we denote by $S_{p,q,r}$ 
the {\em boundary component of rank $r$},
i.e.\ the set of all matrices $z\in\d D_{p,q}$
for which the matrix $I_{q}-z^{*}z$ has rank $r$.
We also write $T=TS_{p,q,r}$, $T^{c}=T^{c}S_{p,q,r}$,
for the tangent and complex tangent spaces and add $'$ to those for $S_{p',q',r'}$.
As our second main result, we prove:

\begin{theorem}\Label{main}
Let $f$ be a smooth CR map between open pieces of boundary components $S_{p,q,r}$ and $S_{p',q',r'}$ of rank $r<q$ and $r'$ respectively of bounded symmetric domains $D_{p,q}$ and $D_{p',q'}$,
 such that $df(\xi)\in T'\setminus {T'} ^c$ for any tangent vector $\xi\in T\setminus T^c$.
Assume that
\begin{equation}\Label{main-ineq}
p'-r'<2(p-r),\quad q'-r'<p-r.
\end{equation}
Then $r\leq  r'$ and after composing with suitable automorphisms of $D_{p,q}$ and $D_{p',q'}$,
$f$ takes the block matrix form
\begin{equation}\Label{f}
f(z)=
\begin{pmatrix}
 z  & 0 &0\\
      0 &  I_{r'-r} & 0\\
 0 & 0 & h(z)
\end{pmatrix},
\end{equation}
where $h\colon S_{p,q,r}\to\C^{[(q'-r')-(q-r)] \times [(p'-r')-(p-r)]}$
is a CR map satisfyng
\begin{equation}\Label{h-cond}
Id- h(z)^* h(z)>0.
\end{equation}
Vice versa, for any CR map $h$ satisfying \eqref{h-cond},
$f$ given by \eqref{f} defines a CR map between open pieces of 
$S_{p,q,r}$ and $S_{p',q',r'}$.
\end{theorem}

Comparing to Shilov boundaries $S_{p,q}=S_{p,q,q}$ considered in \cite{KZ},
the lower rank boundary components $S_{p,q,r}$, $r<q$, present the new substantial difficulty by being 
{\em Levi-degenerate}. As a result, similar technique does not lead to desired rigidity. 
In order to overcome this difficulty, we have to employ the higher order nondegeneracy ($2$-nondegeneracy)
involving components of different degree, which requires a different approach.

The proofs of Theorems~\ref{cor0} and \ref{main} are completed in \S\ref{proofs}.

%\begin{theorem}\Label{proper}
%Let $f$ be a smooth CR map between open pieces of $S_{p,q,r}$ and $S_{p',q',r'}$ such that $f(D_{p,q})\subset D_{p',q'}$ and $f(S_{p,q,s})\not\subset S_{p',q',r'}$ for any $s<r$.
%Assume that
%\begin{equation}\Label{main-ineq}
%p'-r'<2(p-r),\quad q'-r'<p-r.
%\end{equation}
%Then $r= r'$ and after composing with suitable automorphisms of $D_{p,q}$ and $D_{p',q'}$,
%$f$ is given by the block matrix
%$$z\mapsto
%\begin{pmatrix}
% z & 0 \\
% 0 & h(z)
%\end{pmatrix},
%$$
%where $h$ is a CR map into  $\C^{q'-q}\times\C^{p'-p}$ such that
%$$Id-h^* h>0.$$
%\end{theorem}

%\section{$q=1$}

\section{Geometry of boundary components}
We shall consider the standard inclusion $D_{p,q}\subset\C^{p\times q}\subset Gr(q,p+q)$,
where $Gr(q,p+q)$ is the Grassmanian of all $q$-dimensional subspaces ($q$-planes) of $\C^{p+q}$.
Here the matrix $z\in \C^{p\times q}$ is identified with the graph in $\C^{p+q}$ of the linear map defined by $z$.
We equip the space $\C^{p+q}$ with the
nondegenerate hermitian form
\begin{equation}\Label{form}
\langle z,w\rangle =  \sum_{j} \eps_{j}z_{j}\bar w_{j},
\quad \eps_{j}=
\begin{cases}
-1, & j=1,\ldots,q,\\
1, & j=q+1,\ldots,q+p,
\end{cases} 
\end{equation}
called the {\em basic form}.

In this identification, $D_{p,q}$ is represented by 
all $q$-planes $V\subset\C^{p+q}$
such that the restriction $\langle\cdot,\cdot\rangle|_{V}$
is negative definite,
and the boundary component $S_{p,q,r}\subset \d D_{p,q}$ of rank $r$
by all $q$-planes $V\subset\C^{p+q}$
such that restriction $\langle\cdot,\cdot\rangle|_{V}$
has $q-r$ negative and $r$ zero eigenvalues.
For $V\in S_{p,q,r}$, denote by $V_0\subset V$ the $r$-dimensional kernel
of $\langle\cdot,\cdot\rangle|_{V}$.
The connected identity component $G$ of the biholomorphic automorphism group $\Aut(D_{p,q})$ is now identified with
the group of all linear transformations of $\C^{p+q}$ preserving $\langle\cdot,\cdot\rangle$,
and each $S_{p,q,r}$ is a $G$-orbit.
In this section we will construct a frame bundle over $S_{p,q,r}$
associated with the CR structure of $S_{p,q,r}$ using Grassmannian frames of $Gr(q,p+q)$.

\subsection{Adapted frames}\Label{adapted}
An adapted $S_{p,q,r}$-frame is a set of vectors
$$Z_{1},\ldots, Z_{r}, Z'_1,\ldots, Z'_{q-r}, X_{1},\ldots, X_{n}, Y_{1},\ldots, Y_{r},$$
where
$$n:=p-r,$$
for which the basic form
is given by the matrix
$$
\begin{pmatrix}
0&0& 0& I_{r}\\
0&-I_{q-r}&0&0\\
0&0&I_{n}&0\\
I_{r}&0&0&0\\
\end{pmatrix}.
$$
Thus we have
$$V_0=\sp \{Z_{1},\ldots,Z_{r}\},
\quad V= V_0 \oplus \sp \{Z'_{1},\ldots, Z'_{q-r}\}$$
and denote
$$V':=\sp\{Z'_{1},\ldots,Z'_{q-r}\}, \quad
X:=\sp \{X_{1},\ldots,X_{n}\}, \quad
Y:=\sp \{ Y_{1},\ldots, Y_{r} \}.$$
The basic form defines the natural duality pairings
$V_0\times Y\to \C$, $V'\times V'\to \C$, $X\times X\to \C$, i.e.\ we have the identifications
\begin{equation}\Label{duals}
\1{V_0}\cong Y^{*}, \quad \1{V'}\cong V'^{*}, \quad \1X\cong X^{*},
\end{equation}
where the ``bar'' over a complex vector space
always denotes the same real vector space with
the negative complex structure.

\subsection{The tangent space of $S_{p,q,r}$}
The tangent space to the Grassmanian $G_{p,q}$ of all
$q$-dimensional subspaces in $\C^{p+q}$ at the element $V$
is isomorphic to $\Hom(V,\C^{p+q}/V)$.
Hence, given an adapted frame $(Z,Z',X,Y)$, it is isomorphic to
$$T_{V}G_{p,q}=\Hom(V,X\oplus Y).$$
Taking into account the splitting $V=V_0\oplus V'$,
the elements of
$$\Hom(V,X\oplus Y)=\Hom(V_0\oplus V',X\oplus Y)$$
are given by block $2\times 2$ matrices decomposed as
\begin{equation}
R\in
\begin{pmatrix}
\Hom(V_0,X) & \Hom(V_0,Y)\\
\Hom(V',X) & \Hom(V',Y)
\end{pmatrix}.
\end{equation}
Then the real tangent space $T_{V}S_{p,q,r}$ to $S_{p,q,r}$
is
\begin{equation}\Label{decomp}
T=T_{V}S_{p,q,r}=
\begin{pmatrix}
*& \2R\\
* &*
\end{pmatrix},
\quad \2R=-\2R^{*},
\end{equation}
the complex tangent subspace is
\begin{equation}
T^{c}=
\begin{pmatrix}
*& 0\\
* &*
\end{pmatrix}.
\end{equation}
%Indeed, differentiating
%$\langle Z_{\a}, Z_{\b}\rangle =0$
%along $S$ we obtain
%$$\phi_{\a}$$

The complex tangent space $T^{c}$
contains further two invariantly defined subspaces
\begin{equation}
T^{-}:= \{ R\in T^{c} : R(V_0)\subset V\}=
\begin{pmatrix}
0& 0\\
* &*
\end{pmatrix}, \quad
T^{+}:= \{ R\in T^{c} : \langle R(V),V_0 \rangle=0 \}=
\begin{pmatrix}
*& 0\\
* &0
\end{pmatrix},
\end{equation}
such that
$$T^{+}\cap T^{-} = T^{0}, \quad T^{+}+T^{-}=T^{c}.$$

\subsection{The connection matrix form}
Write $S:=S_{p,q,r}$ and
denote by $\6B\to S$ the adapted frame bundle
and by $\pi$ the Maurer-Cartan (connection) form on $\6B$
satisfying the symmetry $\pi^{*}=-\pi$ and
the structure equation $d\pi= \pi\wedge \pi$.
Then we can write
\begin{equation}\Label{pi}
\begin{pmatrix}
dZ_{\a}\\
dZ'_{u}\\
dX_{k}\\
dY_{\a}
\end{pmatrix}
= \pi
\begin{pmatrix}
Z_{\b}\\
Z'_{v}\\
X_{j}\\
Y_{\b}
\end{pmatrix}
=
\begin{pmatrix}
\psi_{\a}^{~\b} & \theta_{\a}^{~v} & \theta_{\a}^{~j} & \phi_{\a}^{~\b}\\
\sigma_{u}^{~\b} &  \omega_{u}^{~v} & \delta_{u}^{~j} & \theta_{u}^{~\b} \\
\sigma_{k}^{~\b} & \delta_{k}^{~v} & \omega_{k}^{~j} & \theta_{k}^{~\b}\\
\xi_{\a}^{~\b} & \sigma_{\a}^{~v} & \sigma_{\a}^{~j} & \3\psi_{\a}^{~\b}\\
\end{pmatrix}
\begin{pmatrix}
Z_{\b}\\
Z'_{v}\\
X_{j}\\
Y_{\b}
\end{pmatrix},
\end{equation}
where the matrix $\pi$ satisfies the symmetry relation
\begin{equation}\Label{sym-rel}
\begin{pmatrix}
\psi_{\a}^{~\b} & \theta_{\a}^{~v} & \theta_{\a}^{~j} & \phi_{\a}^{~\b}\\
\sigma_{u}^{~\b} &  \omega_{u}^{~v} & \delta_{u}^{~j} & \theta_{u}^{~\b} \\
\sigma_{k}^{~\b} & \delta_{k}^{~v} & \omega_{k}^{~j} & \theta_{k}^{~\b}\\
\xi_{\a}^{~\b} & \sigma_{\a}^{~v} & \sigma_{\a}^{~j} & \3\psi_{\a}^{~\b}\\
\end{pmatrix}
=-
\begin{pmatrix}
\3\psi_{\bar\b}^{~\bar\a} & \e_{v}\theta_{\bar v}^{~\bar\a} &  \e_{j}\theta_{\bar j}^{~\bar\a} & \phi_{\bar\b}^{~\bar\a}\\
\e_{u}\sigma_{\bar \b}^{~\bar u} &  \e_{u}\e_{v} \omega_{\bar v}^{~\bar u} & \e_{u}\e_{j}\delta_{\bar j}^{~\bar u} & \e_{u}\theta_{\bar \b}^{~\bar u} \\
\e_{k}\sigma^{~\bar k}_{\bar\b} &  \e_{k}\e_{v}\delta_{\bar v}^{~\bar k} & \e_{k}\e_{j}\omega_{\bar j}^{~\bar k} & \e_{k}\theta_{\bar \b}^{~\bar k}\\
\xi_{\bar\b}^{~\bar\a} &  \e_{v}\sigma_{\bar v}^{~\bar \a} & \e_{j}\sigma_{\bar j}^{~\bar \a} & \psi_{\bar\b}^{~\bar\a}\\
\end{pmatrix},
\end{equation}
where
\begin{equation}\Label{signs}
\e_{u}:=\langle Z'_{u}, Z'_{u}\rangle =-1, \quad u=1,\ldots,q-r, \quad
\e_{j}:=\langle X_{j}, X_{j}\rangle=1, \quad j=1,\ldots,n.
\end{equation}
For instance, differentiating $\langle Z_{\a},X_{j}\rangle=0$ we obtain
$$\langle dZ_{\a},X_{j}\rangle + \langle Z_{\a}, dX_{j}\rangle =0$$
implying
$$\theta_{\a}^{~j}\langle X_{j}, X_{j}\rangle
+\theta_{j}^{~\a}\langle Z_{\a}, Y_{\a}\rangle=0$$
and hence
$$\theta_{\a}^{~j}=-\e_{j}\, \theta_{\bar j}^{~\bar\a}.$$

In the sequel, as in \cite{KZ}, we shall always work with a local section
of the frame bundle $\6B\to S$ and routinely identify forms on
$\6B$ with their pullbacks to $S$ via that section.
With that identification in mind,
the forms $\phi_{\a}^{~\b}$ give a basis
in the space of all {\em contact forms},
i.e.\ forms vanishing on $T^{c}$.
Furthermore, the upper right block forms
$$\begin{pmatrix}
\theta_{\a}^{~j} & \phi_{\a}^{~\b}\\
\delta_{u}^{~k} & \theta_{u}^{~\b}
\end{pmatrix}
$$
give together a basis
in the space of all $(1,0)$ forms on $S$.

\medskip
We shall employ several types of frame changes.

\begin{definition}\Label{changes}
{\rm We call a change of frame}
\begin{enumerate}
\item[i)]change of position {\rm if}
$$
\widetilde Z_\alpha=W_\alpha^{~\beta}Z_\beta,\quad \widetilde Z'_u=W_u^{~\beta}Z_\b+W_u^{~v}Z'_v, \quad
\widetilde Y_\alpha=V_\alpha^{~\beta}Y_\beta+V_\a^{~v}Z'_v,\quad
\widetilde X_j=X_j,
$$
{\rm where $W_0=(W_\alpha^{~\beta})$ and $V_0=(V_\alpha^{~\beta})$ are
$r\times r$ matrices satisfying $V_0^*W_0=I_r$, $W'=(W_u^{~v})$ is a $(q-r)\times(q-r)$ matrix satisfying $W'^*W'=I_{q-r}$ and $V_\a^{~\b}{W^{*}}_\b^{~\gamma}+V_\a^{~v}{W^{*}}_v^{~\gamma}=0$};

\item[ii)]change of real vectors {\rm if}
$$
\widetilde Z_\alpha=Z_\alpha,\quad \widetilde Z'_u=Z'_u,\quad
\widetilde X_j=X_j,\quad
\widetilde Y_\alpha=Y_\alpha+H_\alpha^{~\beta}Z_\beta,
$$
or
\begin{equation}
\begin{pmatrix}
\2Z_{\a}\\
\2Z'_u\\
\2X_{j}\\
\2Y_{\a}
\end{pmatrix}
=
\begin{pmatrix}
I_{r}& 0 & 0 & 0\\
0 & I_{q-r} &0&0\\
0 & 0 & I_{n} & 0\\
H_{\a}^{~\b}&0&0& I_{r}
\end{pmatrix}
\begin{pmatrix}
Z_{\b}\\
Z'_v\\
X_{k}\\
Y_{\b}
\end{pmatrix},
\end{equation}
{\rm where $H=(H_\alpha^{~\beta})$ is a skew hermitian matrix};

\item[iii)]dilation {\rm if}
$$
\widetilde Z_\alpha=\lambda_{\alpha}^{-1}Z_\alpha,\quad \widetilde Z'_u=Z'_u, \quad
\widetilde Y_\alpha=\lambda_\alpha Y_\alpha,\quad
\widetilde X_j=X_j,
$$
{\rm where $\lambda_\alpha>0$};

\item[iv)]rotation {\rm if}
$$
\widetilde Z_\alpha=Z_\alpha,\quad \widetilde Z'_u=Z'_u,\quad
\widetilde Y_\alpha=Y_\alpha,\quad
\widetilde X_j=U_j^{~k}X_k,
$$
{\rm where $(U_j^{~k})$ is a unitary matrix.}
\end{enumerate}
\end{definition}
\medskip

Consider a change of position
as in Definition~\ref{changes}.
Then
$\phi$,
$\theta$ and $\delta$ change to
\begin{align}
\widetilde
\phi_\alpha^{~\beta}&=W_\alpha^{~\gamma}\phi_\gamma^{~\delta}W^{*}{}_{\delta}^{~\b},
\\
\widetilde\theta_\alpha^{~j}&=W_\alpha^{~\beta}\theta_\beta^{~j},\\
\widetilde\theta_u^{~\a}&=
W_u^{~v}\theta_v^{~\b}W^{*}{}_{\b}^{~\a}+W_u^{~\b}\phi_\b^{~\gamma}W^{*}{}_\gamma^{~\a}\\
\widetilde\delta_\a^{~j}&=W_u^{~v}\delta_v^{~j},
\end{align}
where $W^{*}{}_{\delta}^{~\b}=\overline{W_{\beta}^{~\delta}}$.
We shall also make use of the change of frame given by
$$
\widetilde Z_\alpha=Z_\alpha,\quad \widetilde Z'_u=Z'_u,\quad
\widetilde X_j=X_{j} + C_j^{~\beta}Z_\beta,\quad
\widetilde Y_\alpha=Y_\alpha+A_\alpha^{~\beta}Z_\beta+B_\alpha^{~j}X_j,
$$
or
\begin{equation}
\begin{pmatrix}
\2Z_{\a}\\
\2Z'_u\\
\2X_{j}\\
\2Y_{\a}
\end{pmatrix}
=
\begin{pmatrix}
I_{r} & 0 & 0 &0\\
0 & I_{q-r} & 0 & 0\\
C_{j}^{~\b}&0 & I_{n} & 0\\
A_{\a}^{~\b}& 0& B_{\a}^{~j}& I_{r}
\end{pmatrix}
\begin{pmatrix}
Z_{\b}\\
Z'_u\\
X_{k}\\
Y_{\b}
\end{pmatrix},
\end{equation}
such that
$$C_j^{~\alpha}+B_j^{~\alpha}=0$$
and
$$(A_\alpha^{~\beta} + \overline{A_\beta^{~\alpha}})
+B_\alpha^{~j}B_j^{~\beta}=0,$$
where
$$B_j^{~\alpha}:=\overline{B_\alpha^{~j}}.$$
Then the new frame $(\widetilde Z,\widetilde Z',\widetilde Y,\widetilde X)$ is an $S_{p,q,r}$-frame. In fact,
\begin{multline}
0=\langle \2Y_{\a},\2Y_{\b}\rangle=
\langle Y_\alpha+A_\alpha^{~\delta}Z_\delta+B_\alpha^{~j}X_j,
Y_\b+A_\b^{~\g}Z_\g+B_\b^{~k}X_k \rangle \\
=
 A_{\a}^{~\b} \langle Z_{\b}, Y_{\b}\rangle
 + \1{A_{\b}^{~\a}} \langle Y_{\a}, Z_{\a}\rangle
 +\sum_{j}B_{\a}^{~j}\1{B_{\b}^{j}}\langle X_{j}, X_{j}\rangle
 = (A_{\a}^{~\b} + \1{A_{\b}^{~\a}})+ \sum_{j }B_{\a}^{~j}\1{B_{\b}^{j}},
\end{multline}
and
\begin{multline}
0=\langle \2X_{j}, \2Y_{\a}\rangle=
\langle X_{j} + C_j^{~\beta}Z_\beta ,
Y_\alpha+A_\alpha^{~\delta}Z_\delta+B_\alpha^{~k}X_k
\rangle
=
 C_{j}^{~\a} \langle Z_{\a}, Y_{\a}\rangle
 + \1{B_{\a}^{~j}} \langle X_{j}, X_{j}\rangle
  = C_{j}^{~\a}
 + \1{B_{\a}^{~j}},
\end{multline}
whereas the other scalar products are obviously zero.
Furthermore, we claim that the related $1$-forms $\widetilde\phi_\alpha^{~\beta}$ and $\widetilde\theta_u^{~\a}$ remain the same, while $\widetilde\theta_\alpha^{~j}$ and $\widetilde\delta_\a^{~j}$ change to
\begin{align}
\widetilde\theta_\alpha^{~j}&=\theta_\alpha^{~j}-\phi_\alpha^{~\beta}B_\beta^{~j},\\
\widetilde\delta_u^{~j}&=\delta_u^{~j}-\theta_u^{~\beta}B_\beta^{~j}.
\end{align}
Indeed, differentiation yields
\begin{align}
d\2Z_{\a}& =
\2\psi_{\a}^{~\b}\2Z_{\b} + \2\theta_{\a}^{~v}\2Z'_v+\2\theta_{\a}^{~j}\2X_{j} +  \2\phi_{\a}^{~\b} \2Y_{\b}\\
&=
\2\psi_{\a}^{~\b}Z_{\b} + \2\theta_{\a}^{~v}Z'_v+\2\theta_{\a}^{~j}(X_{j} + C_j^{~\beta}Z_\beta) +  \2\phi_{\a}^{~\b} (Y_\b+A_\b^{~\g}Z_\g+B_\b^{~j}X_j)\\
 &=dZ_{\a}=\psi_{\a}^{~\b}Z_{\b} +\theta_{\a}^{~v}Z'_v+ \theta_{\a}^{~j}X_{j} +  \phi_{\a}^{~\b} Y_{\b}
\end{align}
and
\begin{align}
d\2Z_{u}& =
\2\sigma_{u}^{~\b}\2Z_{\b} + \2\omega_{u}^{~v}\2Z'_v+\2\delta_{u}^{~j}\2X_{j} +  \2\theta_{u}^{~\b} \2Y_{\b}\\
&=
\2\sigma_{u}^{~\b}Z_{\b} + \2\omega_{u}^{~v}Z'_v+\2\delta_{u}^{~j}(X_{j} + C_j^{~\beta}Z_\beta) +  \2\theta_{u}^{~\b} (Y_\b+A_\b^{~\g}Z_\g+B_\b^{~j}X_j)\\
 &=dZ_{u}=\sigma_{u}^{~\b}Z_{\b} +\omega_{u}^{~v}Z'_v+ \delta_{u}^{~j}X_{j} +  \theta_{u}^{~\b} Y_{\b}
\end{align}
and the claim follows from identifying the coefficients.
\medskip

\subsection{Structure identities}
The structure equations yield
\begin{equation}\Label{struc-phi}
d\phi_{\a}^{~\b}= \theta_{\a}^{~j}\wedge \theta_{j}^{~\b} +
\theta_{\a}^{~u}\wedge \theta_{u}^{~\b} \mod \phi,
\end{equation}
\begin{equation}\Label{struc-theta1}
d\theta_{\a}^{~j}= \theta_{\a}^{~v}\wedge \delta_{v}^{~j}  \mod
\{\theta_{\b}^{~k}, \phi\},
\end{equation}
\begin{equation}\Label{struc-theta2}
d\theta_{u}^{~\b}= \delta_{u}^{~k}\wedge \theta_{k}^{~\b}  \mod
\{\theta_{v}^{~\a}, \phi\},
\end{equation}
where $\phi$ stands for the span of all $\phi_\a^{~\b}$.
%where
%$$\theta_{\a}$$
The first one via Cartan's formula
$$d\tau(R_{1},R_{2})=R_{1}\tau(R_{2})-R_{2}\tau(R_{1})-\tau([R_{1},R_{2}]),$$
determines the invariant tensor
\begin{equation}
\6L=\6L_{1}\colon T^{1,0}\times T^{1,0} \to \frac{\C T}{T^{1,0}+T^{0,1}},  \quad (R_{1},R_{2})\mapsto [R_{1},\1R_{2}] \mod T^{1,0}+T^{0,1},
\end{equation}
which, in the decomposition \eqref{decomp}, takes the form
\begin{equation}
\left(
\begin{pmatrix}
a_{1}&0\\
c_{1}&d_{1}
\end{pmatrix},
\begin{pmatrix}
a_{2}&0\\
c_{2}&d_{2}
\end{pmatrix}
\right) \mapsto
a_{2}^{*}a_{1} - d_{1}d_{2}^{*} \in \Hom (V_0,Y),
\end{equation}
and represents the Levi form of $S$ up to imaginary constant.
In particular,
\begin{equation}
K:=
\begin{pmatrix}
0& 0\\
* &0
\end{pmatrix}
\subset T^{c}
\end{equation}
is the kernel of the Levi form of $S$.
In more invariant terms, $\6L_{1}$ splits into the sum of two
tensors
\begin{equation}\Label{L-e1}
\Hom (V_0,X)\times \Hom(V_0,X) \to \Hom (V_0)\otimes \1{V_0},\C),
\quad (a_{1},a_{2})\mapsto \langle a_{1}, a_{2}\rangle,\end{equation}
\begin{equation}\Label{L-e2}
\Hom (V',Y)\times \Hom(V',Y) \to \Hom (V_0\otimes \1{V_0},\C),
\quad (d_{1},d_{2})\mapsto -\langle d_{2}^{*}, d_{1}^{*}\rangle,
\end{equation}
where we have used the identifications \eqref{duals}.

Similarly,
\eqref{struc-theta1} and \eqref{struc-theta2} determine together
the  invariant tensor
\begin{equation}\Label{2nd}
\6L_{2}\colon K^{1,0}\times T^{1,0}\to \frac{T^{1,0}}{K^{1,0}} \cong \frac{T^{1,0}+T^{0,1}}{K^{1,0}+T^{0,1}},
\quad (R_{1},R_{2})\mapsto [R_{1},\1R_{2}] \mod K^{1,0}+T^{0,1}.
\end{equation}
Note that since $K^{1,0}$ is in the (complexified) Levi kernel,
one always has $[R_{1},\1R_{2}]\subset T^{1,0}+T^{0,1}$.
The tensor $\6L_{2}$ can be regarded as the ``second order Levi form''
that comes naturally into consideration along with the (first order) Levi form $\6L_{1}$ to gain the ``missing nondegeneracy''.
In the decomposition \eqref{decomp}, $\6L_{2}$ takes the form
\begin{equation}
\left(
\begin{pmatrix}
0&0\\
c_{1}&0
\end{pmatrix},
\begin{pmatrix}
a_{2}&0\\
c_{2}&d_{2}
\end{pmatrix}
\right) \mapsto
(-c_{1 }d_{2}^*) \oplus a_{2}^{*} c_{1} \in \Hom (V_0,X)\oplus\Hom(V',Y),
\end{equation}
or, in more invariant terms,
splits into the sum of two tensors
\begin{equation}
\Hom (V',X)\times \Hom(V',Y) \to \Hom (V_0\otimes \1{X},\C),
\quad (c_{1},d_{2})\mapsto - \langle d_{2}^{*}, c_{1}^{*}\rangle,\end{equation}
\begin{equation}
\Hom (V',X)\times \Hom(V_0,X) \to \Hom (V'\otimes \1{V_0},\C),
\quad (c_{1},a_{2})\mapsto \langle c_{1}, a_{2}\rangle,
\end{equation}

\subsection{Important special cases}
The case
$$q\le p, \quad r=0,$$
corresponds to the Grassmanian of all ``maximal negative definite subspaces'',
which is the bounded symmetric domain of type $I_{p,q}$,
where $q$ is the rank.
More generally, the case
$$0< r\leq q$$
corresponds to the rank $r$ boundary component
of the above bounded symmetric domain.
%The latter case is also characterized by the properties
%$$\langle Z'_{u}, Z'_{u}\rangle = \e_{u}=-1, \quad
%\langle X_{j}, X_{j}\rangle = \e_{j}= 1.$$
Then, in view of \eqref{L-e1} and \eqref{L-e2},
the tensor $\6L$ can be represented by the sesqui-linear map
\begin{equation}
\Hom(V_0,V'\oplus X) \times  \Hom(V_0,V'\oplus X) \to
\Hom(V_0\otimes \1{V_0},\C), \quad (h_{1},h_{2})\mapsto \langle h_{1}, h_{2}\rangle_{0},
\end{equation}
where $\langle\cdot,\cdot\rangle_{0}$
is the standard positive definite hermitian form
making the basis $Z'_{u},X_{j}$ orthonormal.

\subsection{Structure tensor identities for CR-maps}

Let $M=S_{p,q,r}$ and $M'=S_{p',q',r'}$. We shall consider a CR-map $f\colon M\to M'$,
write latin $a,b,c,\ldots$ instead of Greek $\a,\b,\g,\ldots$, and capital instead of small roman letters for the connection forms on $M'$
and as in \cite{KZ}, by slight abuse of notation,
use the same letters to denote pullbacks of these forms to $M$ via $f$.
The structure equation \eqref{struc-phi}
and its analogue for $M'$ imply the equivariance identity
for the first structure tensors:
\begin{equation}\Label{L-eq}
f_{*}\6L_{1}(R_{1},R_{2})= \6L'_{1}(f_{*}R_{1},f_{*}R_{2}).
\end{equation}

As before we identify the complexified normal space
$\C T / (T^{1,0}+ T^{0,1})$ with
$\Hom (V_0\otimes \1{V_0}, \C)$, i.e.\
with the space of all sesqui-linear forms on $V_0$.
Those forms are spanned by the rank one forms
$\mu\otimes\1\mu$, where $\mu\colon V_0\to\C$ is a
complex-linear functional.

Choose any complex-linear functional $\mu\colon V'\to\C$
such that
\begin{equation}\Label{nonvanish}
f_{*}(\mu\otimes\1\mu)\ne0 \in \C TS'/(T^{1,0}S'+ T^{0,1}S').
\end{equation}
Then for the given frame $Z_{\a}, Z'_{u}, X_{j}, Y_{\a}$ on $S$,
the rank $1$ homomorphisms
$$\mu Z'_{u}\in \Hom(V',V''), \quad \mu X_{j}\in \Hom(V',X)$$
yield tangent vectors
$$
Z'^{\mu}_{u}:=
\begin{pmatrix}
0&0\\
0& (\mu Z'_{u})^{*}
\end{pmatrix},
\quad
X_{j}^{\mu}:=
\begin{pmatrix}
\mu X_{j}&0\\
0&0
\end{pmatrix},
$$
which are in view of  \eqref{L-e1} and \eqref{L-e2},
pairwise $\6L$-orthogonal and satisfy
$$\6L(Z'^{\mu}_{u},Z'^{\mu}_{u})=-(\mu\otimes\1\mu) \langle Z'_{u},Z'_{u}\rangle =
-\e_{u} (\mu\otimes\1\mu), \quad
\6L(X_{j}^{\mu},X_{j}^{\mu})=(\mu\otimes\1\mu) \langle X_{j},X_{j}\rangle =
\e_{j} (\mu\otimes\1\mu).$$
In view of \eqref{L-eq},
the push-forwards
$f_{*}Z'^{\mu}_{u}$, $f_{*}X_{j}^{\mu}$
are pairwise $\6L'$-orthogonal and satisfy
\begin{equation}\Label{levi-id}
\6L'(f_{*}Z'^{\mu}_{u},f_{*}Z'^{\mu}_{u})=-\e_{u} f_{*}(\mu\otimes\1\mu), \quad
\6L'(f_{*}X_{j}^{\mu},f_{*}X_{j}^{\mu})=\e_{j} f_{*}(\mu\otimes\1\mu).
\end{equation}

\section{Determination of $\Phi_{a}{}^{b}$}

\subsection{Determination of $\Phi_{1}{}^{1}$}

Choose a diagonal contact form of $M'$ and say
$\Phi_{1}{}^{1}$. Since contact forms are spanned by $\phi_{\a}{}^{\b}$,
we can write
$$\Phi_{1}{}^{1} = c_{\a}{}^{\b} \phi_{\a}{}^{\b}$$
for some smooth functions $c_{\a}{}^{\b}$. At generic points, we may assume that either $c_{\a}{}^{\b}\equiv 0$ or the matrix $(c_{\a}{}^{\b})$ is of constant rank $l\geq 1$.
As in \cite{KZ}, after a unitary change of frame on $M$,
we obtain
\begin{equation*}
\Phi_1^{~1}=\sum_{\alpha=1}^r c_\alpha \phi_\alpha^{~\alpha}
\end{equation*}
for smooth functions $c_\alpha$. If $c_{\a}{}^{\b}\equiv 0$, then $c_\alpha\equiv 0$ for all $\a$ and if the matrix $(c_{\a}{}^{\b})$ has constant rank $l\geq 1$, then we may assume that $c_\alpha,~\alpha=1,\ldots,l,$ never vanish and $c_\alpha\equiv 0$ for $\alpha>l$. Then using \eqref{struc-phi} and its analogue for $M'$ we obtain
\begin{equation}\Label{struc-phi11}
 \Theta_{1}^{~J}\wedge \Theta_{J}^{~1} +
\Theta_{1}^{~U}\wedge \Theta_{U}^{~1}
= \sum_{\a} c_{\a}(
\theta_{\a}^{~j}\wedge \theta_{j}^{~\a} +
\theta_{\a}^{~u}\wedge \theta_{u}^{~\a}
)
\mod \phi,
\end{equation}
Arguing similar to \cite{KZ} we conclude $c_{\a}\geq0$
and, after dilation, $c_{1}=1$ if $c_1\not\equiv 0$.

Along with the span $\phi$ used before we shall use shortcut notation $\theta$ (resp.\ $\Theta$ for $M'$) for the span of the $(1,0)$ forms $\theta_\a^{~j}$, $\theta_u^{~\a}$.
Since in view of \eqref{levi-id}, $f$ sends the Levi kernel of $M$ given by $\phi=\theta=0$
into the Levi kernel of $M'$ given by $\Phi=\Theta=0$, we can write
\begin{align}
\Theta_{1}^{~J}= h_{j}^{J,\a} \theta_{\a}^{~j}   + g_{\a}^{J,u} \theta_{u}^{~\a}
&\mod\phi,  \Label{theta-exp} \\
\Theta_{U}^{~1} = \eta_{U,j}^{\a} \theta_{\a}^{~j} + \xi_{U,\a}^{u} \theta_{u}^{~\a}
&\mod\phi.\Label{theta-u}
\end{align}
Then \eqref{struc-phi11} together with symmetry relations \eqref{sym-rel} implies
\begin{align}
\sum_{J} h_{j}^{J,\a} \1{h_{k}^{J,\b} }  + \sum_{U} \eta_{U,j}^{\a} \1{\eta_{U,k}^{\b} }
&= c_{\a}{\hat\delta}_{\a\b}\hat\delta_{jk}, \Label{h}
\\
\sum_{J} h_{j}^{J,\a} \1{g_{\b}^{J,u} }  + \sum_{U} \eta_{U,j}^{\a} \1{\xi_{U,\b}^{u} }
&= 0, \Label{hg}
\\
\sum_{J} g_{\a}^{J,u} \1{g_{\b}^{J,v} }  + \sum_{U} \xi_{U,\a}^{u} \1{\xi_{U,\b}^{v} }
&= c_{\a}{\hat\delta}_{\a\b}\hat\delta_{uv}, \Label{g}
\end{align}
where $\hat\delta$ is the Kronecker delta. If $c_\alpha\equiv 0$ for all $\alpha$, then from \eqref{h}, \eqref{g} and \eqref{theta-exp}, \eqref{theta-u} we obtain
$$\Theta_1^{~J}=\Theta_{U}^{~1}=0~\mod\phi.$$

Now suppose that $c_1=1$. Substituting \eqref{theta-exp} and \eqref{theta-u} respectively into
 the analogs of \eqref{struc-theta1} and \eqref{struc-theta2} for $M'$ yields
\begin{align}
h_{j}^{J,\a} d\theta_{\a}^{~j}   + g_{\a}^{J,u} d\theta_{u}^{~\a}
=  (\eta_{\a}^{V,j}  \theta_{j}^{~\a}
+ \xi_{u}^{V,\a} \theta_{\a}^{~u})
\wedge \Delta_{V}^{~J}
&\mod
\theta, \phi,  \Label{struc-theta1'}\\
\eta_{U,j}^{\a} d\theta_{\a}^{~j} + \xi_{U,\a}^{u}  d\theta_{u}^{~\a}
= \Delta_{U}^{~J}\wedge  (h_{J,\a}^{j} \theta_{j}^{~\a} + g_{J,u}^{~\a} \theta_{\a}^{~u} )
& \mod
\theta, \phi,
\Label{struc-theta2'}
\end{align}
where
\begin{equation}
\eta_{\a}^{U,j} :=
-\1{\eta_{U,j}^{\a}}, \quad
\xi_{u}^{U,\a} :=
\1{\xi_{U,\a}^{u}},
\quad h_{J,\a}^{j}:=  \1{h_{j}^{J,\a}},
\quad g_{J,u}^{\a}:=- \1{g_{\a}^{J,u}}.
\end{equation}
Using \eqref{struc-theta1} and \eqref{struc-theta2}, we rewrite \eqref{struc-theta1'} and \eqref{struc-theta2'} as
\begin{align}
h_{j}^{J,\a}  \theta_{\a}^{~u}\wedge \delta_{u}^{~j}   + g_{\a}^{J,u} \delta_{u}^{~j}\wedge \theta_{j}^{~\a}
=  (\eta_{\a}^{V,j}  \theta_{j}^{~\a}
+ \xi_{u}^{V,\a} \theta_{\a}^{~u})
\wedge \Delta_{V}^{~J}
&\mod
\theta, \phi,  \Label{struc-theta1''}\\
\eta_{U,j}^{\a} \theta_{\a}^{~u}\wedge \delta_{u}^{~j}
+ \xi_{U,\a}^{u} \delta_{u}^{~j} \wedge \theta_{j}^{~\a}
=\Delta_{U}^{~J}\wedge  (h_{J,\a}^{j} \theta_{j}^{~\a} + g_{J,u}^{~\a} \theta_{\a}^{~u} )
& \mod
\theta, \phi. \Label{struc-theta2''}
\end{align}
By Cartan's Lemma,
\begin{align}
\eta_{\a}^{V,j} \Delta_{V}^{~J}  = - g_{\a}^{J,u} \delta_{u}^{~j} &\mod \theta, \bar\theta, \phi,
\Label{cartan}
\\
g_{J,u}^{\a} \Delta_{U}^{~J}  = - \eta_{U,j}^{\a} \delta_{u}^{~j} &\mod \theta, \bar\theta, \phi,
\Label{cartan1}
\end{align}

For $\a$, $j$ fixed, consider vector $\eta_{\a}^{j}:=(\eta_{\a}^{U,j})_{U}\in \C^{q'-r'}$.
Since $q'-r'<n$ by \eqref{main-ineq}, these vectors are linearly dependent, i.e.\
$ \sum d_{j} \eta_{\a}^{j} =0$ for some $(d_{1},\ldots,d_{n})\ne0$.
Then by \eqref{cartan},
\begin{equation}
 g_{\a}^{J,u} d_{j}\delta_{u}^{~j} =0\mod \theta, \bar\theta, \phi.
\end{equation}
Since $q-r\ge 1$ and $\delta_{u}^{~j}$, $1\le u\le q-r$, $1\le j\le n$, are linearly independent modulo $\theta, \bar\theta, \phi$,
it follows that $ g_{\a}^{J,u}=0$, and hence $\eta_{U,j}^{\a}=0$ by \eqref{cartan1}.

Now it follows from \eqref{h} that the vectors $h_{j}^{\a}:=(h_{j}^{J,\a})_{J}$
are pairwise orthogonal and have length $c_{\a}$ independent of $j$.
Then after a unitary rotation of the frame as in \cite{KZ} (proof of Lemma~4.1),
we may assume that the vectors $h_{j}^{1}$ of length $c_{1}=1$
are precisely the first $n$ standard vectors in $\C^{n'}$.
Using $n'<2n$ and repeating the rest of the argument in \cite{KZ}, we conclude $s=1$
and hence
\begin{align}
\Phi_{1}^{~1}&=\phi_{1}^{~1},\Label{phi11-norm} \\
\Theta_1^{~J}&=\theta_{1}^{~J} \mod \phi.\Label{theta-norm}
\end{align}

Finally \eqref{g} is now of the form
$$
\sum_{U} \xi_{U,\a}^{u} \1{\xi_{U,\b}^{v} }
= \hat\delta_{1\a}{\hat\delta}_{\a\b}\hat\delta_{uv}. \Label{g'}
$$
Then after a unitary change of frame we obtain
\begin{equation}
\Theta_{U}^{~1} = \theta_{U}^{~1} \mod \phi.
\end{equation}

Summarizing we obtain
\begin{lemma}\Label{r1}
\begin{align}
\Theta_1^{~J}&=c_1\theta_1^{~J}\mod\phi,\\
\Theta_{U}^{~1}& =c_1\theta_U^{~1}\mod \phi,
\end{align}
where $\Phi_1^{~1}=c_1\phi_1^{~1}$,
and $c_1$ is either $0$ or $1$.
\end{lemma}
Furthermore, by considering $\Phi_a^{~a}$ for arbitrary $a$, we can show the following lemma.
\begin{lemma}\Label{r2} Let $\theta^+$ and $\theta^-$ be ideals generated by $\theta_\alpha^{~j}$ and $\theta_u^{~\beta}$ respectively.
Then
\begin{align}
\Theta_a^{~J}&=0\mod \theta^+, \phi,\\
\Theta_{U}^{~a}& =0\mod \theta^-,\phi,
\end{align}
i.e., the subspace $T^+$ and $T^{-}$ are preserved by $f$. That is,
$$f_*(T^+)\subset T'^+,~f_*(T^{-})\subset T'^{-}.$$
\end{lemma}

\subsection{Determination of $\Phi_{2}^{~2}$ and $\Phi_{2}^{~1}$}
Suppose first $\Phi_a^{~a}\equiv 0$ for all $a$. Then Lemma~\ref{r1} and its analogues for $\a=2,\ldots,r$, we obtain
\begin{align}
\Phi&\equiv 0\\
\Theta&\equiv 0~\mod \phi.
\end{align}

Now assume that there exists $a$ such that $\Phi_a^{~a}\not\equiv 0$, say $a=1.$ Then $$\Phi_1^{~1}=\phi_1^{~1}$$ by Lemma~\ref{r1}
and let
\begin{equation}\Label{la}
\Phi_a^{~1}=\lambda_a\phi_1^{~1}~\text{mod}~
\{\phi_\alpha^{~\beta},  \a\ge 2 \text{ or } \b\ge2\}, \quad a\ge2,
\end{equation}
for some smooth functions $\lambda_a$, $a=2,\ldots,r'$.
Then \eqref{struc-phi} and its analogue for $M'$ together with Lemmas~\ref{r1} and ~\ref{r2} imply
\begin{equation}\Label{thaj}
\Theta_a^{~j}\wedge \theta_j^{~1}=\lambda_a\theta_1^{~j}\wedge\theta_j^{~1}~
\mod \theta_\alpha,~ \1{\theta_{\alpha}},~\alpha\ge2,
~\theta^-,~\1{\theta^-},~\phi, \quad a\ge2,
\end{equation}
where $\theta_\a$ is the span of all $\theta_\a^{~j}$.
%,and $\theta_u$ is the span of all $\theta_u^{~\b}$.

Then there exists a change of position (see Definition~\ref{changes}) that leaves $\Theta_{1}^{~J}$ invariant and replaces $\Theta_a^{~J}$
with $\Theta_a^{~J}-\lambda_a\Theta_1^{~J}$, $a\ge2$, (see the discussion after Definition~\ref{changes}).
The same change of position leaves $\Phi_{1}^{~1}$ invariant
and transforms $\Phi_{a}^{~1}$ into $\Phi_{a}^{~1}-\l_{a}\Phi_{1}^{~1}$ for $a\ge2$.
After performing such change of position, \eqref{la} becomes
$$\Phi_a^{~1}=0~\text{mod}~
\{\phi_\alpha^{~\beta} : \a\ge 2 \text{ or } \b\ge2\}, \quad a\ge2,$$
and \eqref{thaj} becomes
\begin{equation}
\Theta_a^{~j}\wedge \theta_j^{~1}+\Theta_a^{~u}\wedge\theta_u^{~1}=0~\mod \theta_\alpha,~ \1{\theta_{\alpha}},~\alpha\ge2,
~\theta^-,~\1{\theta^-},~\phi, \quad a\ge2.
\end{equation}
Since $\Theta_a^{~j}$, $\theta_u^{~1}$ are $(1,0)$ but $\Theta_a^{~u}$, $\theta_j^{~1}$ are $(0,1)$ and linearly independent,
it follows from Cartan's lemma together with Lemma~\ref{r2} that
\begin{equation}
\Theta_a^{~j}=\Theta_u^{~a}=0 \mod  \{\theta_\a,~\theta_v^{~\alpha} : \a\ge2\},\, \phi, \quad a\geq 2.
\end{equation}
Since $\Theta_a^{~j}$ are spanned by $\theta^+$ and $\Theta_u^{~a}$ are spanned by $\theta^-$, we conclude that
\begin{align}
\Theta_a^{~j}&=0 \mod  \{\theta_\a,: \a\ge2\},\, \phi, \quad a\geq 2\Label{theta-aj}\\
\Theta_u^{~a}&=0 \mod  \{\theta_v^{~\a} : \a\ge2\},\, \phi, \quad a\geq 2\Label{theta-ua}.
\end{align}

Next for each $a\ge2$, let
\begin{equation}\Label{lab}
\Phi_a^{~a}=\lambda_{a,\beta}\phi_1^{~\beta}
\mod\{\phi_\alpha^{~\g} : \alpha\geq 2\}
\end{equation}
for some
functions $\lambda_{a,\beta}$. Suppose  first that there
exists $a$ and $\beta$ such that $\lambda_{a,\beta}\neq 0$. We may assume $a=2$.
Using the identity
$$
d\Phi_2^{~2}=~\Theta_2^{~J}\wedge\Theta_J^{~2} \mod \Phi,~\Theta^-, $$
together with \eqref{theta-aj} and \eqref{struc-phi} we obtain
\begin{equation}\Label{2a}
\sum_{J=n+1}^{n'} \Theta_2^{~J}\wedge \Theta^{~2}_J=\lambda_{2,\beta}~ \theta_1^{~j}\wedge\theta_j^{~\beta} \mod \{\theta_\a,~\theta^-: \a\ge2\},\phi,
\end{equation}
where $\l_{2,\b}\ne0$ for some fixed $\b$.
On the left-hand side we have a linear combination of $n'-n$ $(1,0)$ forms,
whereas on the right-hand side we have a linear combination of at least $n$ linear independent $(1,0)$ forms with nonzero coefficients.
Since $n'-n<n$, Cartan's lemma implies that this is impossible.
Hence we have $\l_{a,\b}=0$ for all $a\ge2$ and all $\b$ and therefore \eqref{lab} implies
$$\Phi_a^{~a}=0~\text{   mod
}\{\phi_\alpha^{~\beta} : \alpha\ge 2\}, \quad a\ge2.$$
Since $\Phi_{a}^{~b}$ and $\phi_{\a}^{~\b}$ are antihermitian, we also have
$$\Phi_a^{~a}=0~\text{   mod
}\{\phi_\alpha^{~\beta} : \b\ge 2\}, \quad a\ge2,$$
and hence
\begin{equation}\Label{phi-aa}
\Phi_a^{~a}=0~\text{   mod
}\{\phi_\alpha^{~\beta}: \a,\b\ge2 \}, \quad a\ge 2.
\end{equation}

Now \eqref{struc-phi}, \eqref{theta-aj}, \eqref{theta-ua} and Lemma~\ref{r2} imply
\begin{equation}\Label{aa}
\sum_{J=n+1}^{n'} \Theta_a^{~J}\wedge \1{\Theta_{a}^{~J}}+\sum_{U=q-r+1}^{q'-r'}\1{\Theta_a^{~U}}\wedge\Theta_a^{~U}=0~\text{   mod   }\{\theta_\a , \theta_u^{~\alpha}: \a\ge2\},\phi, \quad a\ge2,
\end{equation}
which in view of Lemma~\ref{r2} and positivity of the left-hand side implies
\begin{align}
\Theta^{~J}_a&=0 ~\text{   mod   }\{\theta_\a : \a\ge2\},\phi,
\quad a\ge2,\, J>n,\Label{aj}\\
\Theta_U^{~a}&=0 ~\text{   mod   }\{\theta_u^{~\a} : \a\ge2\},\phi,
\quad a\ge2,\, U>q-r\Label{ua}.
\end{align}
Together with \eqref{theta-aj} and \eqref{theta-ua} this yields
\begin{align}
\Theta_a^{~J}&=0 \mod  \{\theta_\a : \a\ge2\},\, \phi, \quad a\geq 2,\Label{theta-aJ}\\
\Theta_U^{~a}&=0 ~\text{   mod   }\{\theta_u^{~\a} : \a\ge2\},\phi,
\quad a\ge2\Label{theta-Ua}.
\end{align}

Now we redo our procedure for $\Phi_{a}^{~b}$. We can write
\begin{equation}\Label{phi-ab'}
\Phi_{a}^{~b}= \l_{a}^{~b}{}_{\b}^{~\a} \phi_{\a}^{~\b}
\end{equation}
for which \eqref{struc-phi} yields
\begin{equation}\Label{lambda-a}
\Theta_a^{~J}\wedge \Theta_J^{~b}+\Theta_a^{~U}\wedge \Theta_U^{~b}=
\l_{a}^{~b}{}_{\b}^{~\a}( \theta_\a^{~j}\wedge\theta_j^{~\b}+\theta_\a^{~u}\wedge\theta_u^{~\b})~
\text{mod}~ \phi.
\end{equation}
Then substituting \eqref{theta-aJ} we obtain
\begin{equation}
\l_{a}^{~b}{}_{\b}^{~\a} \theta_\a^{~j}\wedge\theta_j^{~\b} =0
\mod \{\theta_\g^{~k}\wedge\theta_{l}^{~\delta} ,
~\theta_\g^{~u}\wedge \theta_v^{~\delta} : \g,\delta\ge 2\},\phi, \quad a,b\ge2,
\end{equation}
which implies
\begin{equation}
\l_{a}^{~b}{}_{1}^{~\a} = \l_{a}^{~b}{}_{\b}^{~1} = 0, \quad a,b\ge 2.
\end{equation}
Hence \eqref{phi-ab'} yields
\begin{equation}
\Phi_{a}^{~b}= 0 \mod \{\phi_\alpha^{~\beta}: \a,\b\ge2 \}, \quad a,b\ge 2.
\end{equation}
Summarizing we obtain the following:
\begin{align}
\Phi_a^{~1}&=0 \mod
\{\phi_\alpha^{~\beta} : \a\ge 2 \text{ or } \b\ge2\}, \quad a\ge2, \\
\Phi_{a}^{~b}&= 0 \mod \{\phi_\alpha^{~\beta}: \a,\b\ge2 \}, \quad a,b\ge 2,\\
\Theta^{~J}_a &=0 \mod \{\theta_\a: \a\ge2\},\phi, \quad a\ge2,\\
\Theta_U^{~a}&=0 \mod \{\theta_u^{~\a}: \a\ge2,~u>r\},\phi.\Label{before-repeat}
\end{align}

Now repeat the argument from the beginning of this section and assume first that
$\Phi_{a}^{~a}=0$ for all $a\ge 2$.
We obtain
$$\Theta_{a}^{~J}=\Theta_U^{~a}=0 \mod \phi, \quad a\ge2,$$
and hence $d\Phi_a^{~b}$ vanishes on the kernel of all $\theta_{1}^{~j}$, $\theta_u^{~1}$ and $\phi_{\a}^{\b}$.
In this case %\eqref{aj} and \eqref{ua} imply
\eqref{phi-ab'} and \eqref{lambda-a} imply
$$\Phi_a^{~b}=0, \quad a>1\quad\text{or}\quad b>1.$$
%and
%$$\Theta^{~J}_a=\Theta_U^{~a}=0 \mod \phi, \quad a>1.$$

In the remaining case, we assume that $\Phi_{a}^{~a}\ne 0$ for some $a$, say $a=2$.
Then \eqref{phi-aa} implies that, after a change of position as before,
we may assume that
$$\Phi_{2}^{~2}=\sum_{\a\ge 2} c_{\a} \phi_{\a}^{~\a}$$
for some $c_{\a}\ge0$ not all zero.
Then \eqref{struc-phi} yields
\begin{equation}\label{two theta1}
\Theta_2^{~J}\wedge \Theta_J^{~2}+\Theta_2^{~U}\wedge\Theta_U^{~2}=\sum_{\alpha\ge2} c_\alpha\left(\theta_\alpha^{~j}\wedge\theta_j^{~\alpha}+\theta_\alpha^{~u}\wedge\theta_u^{~\a}\right)
~\text{   mod   }~\phi.
\end{equation}
Since the proof of Lemma~\ref{r1} can be repeated for $\Phi_{2}^{~2}$
instead of $\Phi_{1}^{~1}$, we conclude that
the rank of the left-hand side of \eqref{two theta1} restricted
to $T^1$ is $n$.
Therefore, in the right-hand side, only one $c_{\a}$, say $c_{2}$ can be different from zero. After a dilation (see Definition~\ref{changes}), we may assume
$$\Phi_{2}^{~2}= \phi_{2}^{~2}$$
and hence
\begin{equation}\Label{theta2}
\sum_{J}\T_{2}^{~J}\wedge \1{\T_{2}^{~J}}+\sum_U \1{\Theta_U^{~2}}\wedge\Theta_U^{~2} = \sum_{j}\theta_{2}^{~j}\wedge \1{\theta_{2}^{~j}}+\sum_u \1{\theta_u^{~2}}\wedge\theta_u^{~2}
\mod\phi.
\end{equation}

We claim that each $\T_{2}^{~J}$ and $\Theta_U^{~2}$ is a linear combination of only $\theta_{2}^{~j}$ and $\theta_u^{~2}$ modulo $\phi$. Indeed, if $\T_{2}^{~J}$ were a combination of $\theta_{\a}^{~j}$ modulo $\phi$, where some of them
enters with a nonzero coefficient $\l_{\a}$ with $\a\ne 2$,
we would have $\theta_{\a}^{~j}\wedge \1{\theta_{\a}^{~j}}$
entering with positive coefficient
$\ge\l_{\a}\1{\l_{\a}}$ in the right-hand side of \eqref{theta2}, which is impossible. Similar argument for $\Theta_U^{~2}$ proves our claim.
As in the proof of Lemma~\ref{r1} we now write
\begin{equation}\Label{teta2}
\Theta_2^{~J}=h^{J}_{~j}\theta_2^{~j}~\text{   mod   }~\phi.
\end{equation}
Since
\begin{equation}\Label{tet2}
\Phi_{2}^{~1}= \l_{\b}^{~\a} \phi_{\a}^{~\b}
\end{equation}
for suitable $\l_{\a}^{~\b}$, we obtain
\begin{equation}
\T_{2}^{~J}\wedge \T_{J}^{~1} = \l_{\b}^{~\a} \theta_{\a}^{~j} \wedge
\theta_{j}^{~\b} \mod \theta_\g^{~u}\wedge\theta_v^{~\delta}, \phi,
\end{equation}
which in view of \eqref{teta2} and Lemma~\ref{r1}, yields
\begin{equation}\Label{h-id}
h^{k}_{~j}\theta_2^{~j} \wedge \theta_{k}^{~1} = \l_{\b}^{~\a} \theta_{\a}^{~j} \wedge
\theta_{j}^{~\b} \mod \theta_\g^{~u}\wedge\theta_v^{~\delta}, \phi.
\end{equation}
Since the right-hand side contains no terms
$\theta_{\a}^{~j} \wedge
\theta_{k}^{~\b}$ with $j\ne k$, it follows that
$h^{k}_{~j}=0$ for $j\ne k$ and hence $h^{j}_{~j}=\l_{2}^{~1}=:\l$ for all $j$
and $\l_{\b}^{~\a}=0$ for $(\a,\b)\ne(1,2)$.
Then \eqref{teta2} implies
\begin{equation}\Label{teta3}
\Theta_2^{~j}=\l \theta_2^{~j}~\text{   mod   }~\phi.
\end{equation}
Finally, substituting \eqref{teta3} into \eqref{theta2} and identifying coefficients we obtain
$$ \l\bar\l \delta_{ij} + \sum_{J>n} h_{~i}^{J}\1{h_{~j}^{J}}
=\delta_{ij}.$$
In particular, it follows that the vectors $h_{i}:=(h^{n+1}_{~i},\ldots,h^{n'}_{~i})$ are orthogonal and of the same length.
But since we have assumed $n'-n<n$, we must have $h_{i}=0$
and therefore $|\l|=1$. Now we perform a change of position
as in Definition~\ref{changes} with $W_{\a}^{~\b}:=c_{\a}\delta_{\a\b}$ with
$c_{\a}=1$ for $\a\ne2$ and $c_{2}=\l$.
Then we arrive at the following relations:
\begin{equation}
\Phi_{a}^{~a}=\phi_{a}^{~a}, \quad a=1,2,
\end{equation}
\begin{equation}
 \T_{a}^{~J}=\theta_{a}^{~J} \mod \phi, \quad a=1,2.
\end{equation}

Since after the last change of position, we have $\l=1$ in \eqref{teta3},
we obtain from \eqref{h-id} that $\l_{\b}^{~\a}=0$ unless $\a=2$ and $\b=1$,
in which case $\l_{2}^{~1}=1$. Then substituting into \eqref{tet2} yields
\begin{equation}
\Phi_{2}^{~1}=\phi_{2}^{~1}.
\end{equation}
Finally \eqref{struc-phi}, \eqref{g} and Lemma~\ref{r1} imply
\begin{equation}
\Theta_U^{~a}=\theta_U^{~a} \mod \phi, \quad a=1,2.
\end{equation}

\subsection{Determination of $\Phi_{a}^{~b}$}
Now we repeat again the arguments
%of this section following the lines after \eqref{before-repeat},
after the proof of Lemma~\ref{r1},
where we replace $1$ by $2$ and $2$ by $3$, to arrive at the identities:
\begin{align}
\Phi_a^{~2}&=0 \mod
\{\phi_\alpha^{~\beta} : \a\ge 3 \text{ or } \b\ge3\}, \quad a\ge3, \\
\Phi_{a}^{~b}&= 0 \mod \{\phi_\alpha^{~\beta}: \a,\b\ge3 \}, \quad a,b\ge 3,\\
\Theta^{~J}_a &=0 ~\text{   mod   } \{\theta_\a : \a\ge 3\},\,\phi, \quad a\ge3,\\
\Theta_U^{~a} &=0 \mod \{\theta_u^{~\a}: \a\ge 3, u>r\},\phi, \quad a\ge 3.
\end{align}
Then continuing following the arguments after \eqref{before-repeat}
with the same replacements, we obtain
\begin{equation}
\Phi_{3}^{~\a}=\phi_{3}^{~\a},\quad \a=1,2,3,
\end{equation}
\begin{equation}
 \T_{3}^{~J}-\theta_{3}^{~J}= \Theta_U^{~3}-\theta_U^{~3}=0 \mod \phi.
\end{equation}

Finally, arguing by induction on $b=4,\ldots,q'$, and proceeding
by repeating the same arguments with
$1$ replaced by $b$ and $2$ by $b+1$,
we obtain the following lemma.

\begin{lemma}\Label{eds}
For any local CR mapping $f$ from $S_{p,q,r}$ into $S_{p',q',r'}$,
there exist an integer $s\leq \min(r, r')$ and a choice of sections of the bundles
$\6B_{p,q,r}\to S_{p,q,r}$ and $\6B_{p',q',r'}\to S_{p',q',r'}$
such that the pulled back forms satisfy
\begin{align*}
\Phi_a^{~b}-\tilde\phi_a^{~b}&=0,\\
\Theta_a^{~J}-\tilde\theta_a^{~J}&=\Theta_U^{~a}-\tilde\theta_U^{~a}=0 \mod\phi ,
\end{align*}
where
\begin{align}
\tilde\phi_a^{~b}&:=\phi_a^{~b}\quad \text{if}\quad a, b\leq s,\\
\tilde\theta_a^{~J}&:=\theta_a^{~J}\quad\text{if}\quad a\leq s, ~J\leq n,\\
\tilde\theta_U^{~a}&:=\theta_U^{~a}\quad\text{if}\quad a\leq s,~U\leq q-r,
\end{align}
and $0$ otherwise.
\end{lemma}

%
%\begin{remark}\Label{back1}
%A change of section of $\6B_{p,q}\to S_{p,q}$
%(corresponding to a change of frame on $S_{p,q}$)
%has been used in course of the proof.
%However, once Lemma~\ref{eds} has been established,
%one can change the frame on $S_{p,q}$ back to the original one
%together with the corresponding change of the frame on $S_{p',q'}$
%involving only the subframe $(Z_a, X_J, Y_b)$ with $a,b\le q, J\le n$,
%such that the conclusion of the lemma remains valid.
%\end{remark}

\section{Determination of $\Theta$.}

Our next goal is to determine $\Theta_{a}^{~J}$ and $ \Theta_U^{~a}$.
It will be determined together with
components $\Psi$, $\Delta$
and $\Omega$ modulo $\phi$.
In view of Lemma~\ref{eds} we can write
\begin{align}
\Theta_a^{~J}-\tilde\theta_a^{~J}&=
\eta_{a}^{~J}{}_{\b}^{~\g}   \phi_\g^{~\b},\Label{eq1}\\
\Theta_U^{~a}-\tilde\theta_U^{~a}&=
\eta_{U}^{~a}{}_{\b}^{~\g}   \phi_\g^{~\b},\Label{eq2}
\end{align}
for some $\eta_{a}^{~J}{}_{\b}^{~\g},~\eta_{U}^{~a}{}_{\b}^{~\g}$.
% and using the symmetry relations
%\eqref{sym-rel},
%\begin{align}
%\Theta_J^{~a}-\tilde\theta_J^{~a}&=
%\eta_{J}^{~a}{}_{\g}^{~\b}   \phi_\b^{~\g},\Label{eq1'}\\
%\Theta_a^{~U}-\tilde\theta_a^{~U}&=
%\eta_{a}^{~U}{}_{\g}^{~\b}   \phi_\b^{~\g},\Label{eq2'}
%\end{align}
%where
%\begin{equation}
%\eta_{J}^{~a}{}_{\g}^{~\b}:= \1{\eta_{a}^{~J}{}_{\b}^{~\g}},\quad \eta_{a}^{~U}{}_{\g}^{~\b}:= \1{\eta_{U}^{~a}{}_{\b}^{~\g}}.
%\end{equation}

\subsection{Determination of $\Theta_{a}^{~J}, \Theta_U^{~a}$ for $a>s$.}
In case $a>s$, differentiating \eqref{eq1}, \eqref{eq2} and using the structure equations $d\pi=\pi\wedge\pi$ for $M'$, we obtain
\begin{align}
\eta_{a}^{~J}{}_{\b}^{~\g}\left(\theta_\gamma^{~v}\wedge\theta_v^{~\b}+
\theta_\gamma^{~k}\wedge\theta_k^{~\beta}\right)
&=
\Psi_a^{~b} \wedge\tilde\theta_b^{~J}
 \mod \phi, \Label{dtheta1-1}\\
\eta_{U}^{~a}{}_{\b}^{~\g}\left(\theta_\gamma^{~v}\wedge\theta_v^{~\b}+
\theta_\gamma^{~k}\wedge\theta_k^{~\beta}\right)
&=\tilde\theta_U^{~b}\wedge \hat\Psi_b^{~a}
 \mod \phi. \Label{dtheta1-2}
\end{align}

If $J>n$ and $U>q-r$, the right-hand sides of \eqref{dtheta1-1} and \eqref{dtheta1-2} are zero.
Since the forms $\theta_{\g}^{~k},~\theta_u^{~\b}$ and $\theta_{k}^{~\b},~\theta_\g^{~u}$
are $(1,0)$ and $(0,1)$ respectively, and are linearly independent,
 we conclude
\begin{equation*}
\eta_{a}^{~J}{}_{\b}^{~\g}=\eta_{U}^{~a}{}_{\b}^{~\g}=0, \quad a>s, \ J>n, \ U>q-r,
\end{equation*}
and hence \eqref{eq1} and \eqref{eq2} yield
\begin{equation}\Label{theta4}
\Theta_a^{~J}=\Theta_U^{~a}=0,  \quad a>s, \ J>n, \ U>q-r.
\end{equation}

For $J=j\leq n$ and $U=u\leq q-r$, \eqref{dtheta1-1} and \eqref{dtheta1-2} take the form
\begin{align}
\eta_{a}^{~j}{}_{\b}^{~\g}\left(\theta_\gamma^{~v}\wedge\theta_v^{~\b}+
\theta_\gamma^{~k}\wedge\theta_k^{~\beta}\right)
&=
\sum_{b\leq s}\Psi_a^{~b} \wedge \theta_b^{~j}
 \mod \phi, \Label{dtheta1-1'}\\
\eta_{u}^{~a}{}_{\b}^{~\g}\left(\theta_\gamma^{~v}\wedge\theta_v^{~\b}+
\theta_\gamma^{~k}\wedge\theta_k^{~\beta}\right)
&=\sum_{b\leq s}\theta_u^{~b}\wedge \hat\Psi_b^{~a}
 \mod \phi. \Label{dtheta1-2'}
\end{align}
Since the forms $\theta_{\g}^{~k},~\theta_v^{~\b}$ both appear on the left-hand side of \eqref{dtheta1-1'} while only $\theta_\b^{~j}$ appears on the right-hand side of \eqref{dtheta1-1'}, we obtain
$$\eta_{a}^{~j}{}_{\b}^{~\g}=0.$$
Similar argument for \eqref{dtheta1-2'} yields
$$\eta_{u}^{~a}{}_{\b}^{~\g}=0.$$
Hence by \eqref{eq1} and \eqref{eq2}, we obtain
\begin{equation}\Label{theta4'}
\Theta_a^{~j}=\Theta_u^{~a}=0,  \quad a>s.
\end{equation}
Furthermore by substituting $\eta_{a}^{~j}{}_{\b}^{~\g}=\eta_{u}^{~a}{}_{\b}^{~\g}=0$, \eqref{dtheta1-1'} and \eqref{dtheta1-2'} yield
$$\Psi_a^{~b}=0 \mod ~\{\theta_\b^{~j},\phi\},~\quad \text{ if }a>s \text{ and }~b\leq s$$
and
$$\hat\Psi_b^{~a}=0 \mod ~\{\theta_u^{~\b},\phi\},~\quad \text{ if }a>s \text{ and }~b\leq s.$$
Then by symmetry relation for $\Psi$ and $\hat\Psi$, we obtain
\begin{equation}\Label{psi1}
\Psi_a^{~b}=0 \mod ~\phi,~\quad \text{ if }a>s \text{ and }~b\leq s.
\end{equation}

\subsection{Reducing the freedom for $\Theta_{a}^{~J},~\Theta_U^{~a}$ for $a\leq s$}
In case $a=\a\leq s$, differentiating \eqref{eq1}, \eqref{eq2} and using the structure equations $d\pi=\pi\wedge\pi$ for both $M$ and $M'$, we obtain
\begin{align}
&\eta_{\a}^{~J}{}_{\b}^{~\g}\left(\theta_\gamma^{~v}\wedge\theta_v^{~\b}+
\theta_\gamma^{~k}\wedge\theta_k^{~\beta}\right)
+\sum_{\b>s}\psi_\a^{~\b}\wedge\theta_\b^{~J}\nonumber\\
=&
\sum_{\b\leq s}\left(\Psi_\a^{~\b} -\psi_\a^{~\b}\right)\wedge\theta_\b^{~J}+\theta_\a^{~v}\wedge\left(\Delta_v^{~J}-\delta_v^{~J}\right)
+\theta_\a^{~k}\wedge \left(\Omega_k^{~J}-\omega_k^{~J}\right)
 \mod \phi, \Label{dtheta2-1}\\
&\eta_{U}^{~\a}{}_{\b}^{~\g}\left(\theta_\gamma^{~v}\wedge\theta_v^{~\b}+
\theta_\gamma^{~k}\wedge\theta_k^{~\beta}\right)+\sum_{\b>s}\theta_U^{~\b}\wedge\hat\psi_\b^{~\a}\nonumber\\
=&\left(\Omega_U^{~v}-\omega_U^{~v}\right) \wedge\theta_v^{~\a}+
\left(\Delta_U^{~k}-\delta_U^{~k}\right)\wedge\theta_k^{~\a}
+\sum_{\b\leq s}\theta_U^{~\b}\wedge\left( \hat\Psi_\b^{~\a}-\hat\psi_\b^{~\a}\right)
 \mod \phi. \Label{dtheta2-2}
\end{align}
Since the forms $\theta_{\g}^{~u}$ and $\theta_{\b}^{~k}$
are $(0,1)$ and $(1,0)$ respectively and are linearly independent,
the terms $\theta_\gamma^{~v}\wedge\theta_v^{~\beta}$, $\g\ne \a$, in the left-hand of \eqref{dtheta2-1} side cannot occur in the right-hand side. Therefore
\begin{equation}\Label{ag}
\eta_{\a~\beta}^{~J~\gamma}=0~\text{   if   }~\g\ne \a
\end{equation}
and hence \eqref{eq1} becomes
\begin{equation}\Label{eq5-1}
\Theta_\a^{~J} - \tilde \theta_\a^{~J}=\eta_{\alpha~\beta}^{~J}\phi_\alpha^{~\beta}, \quad \a\leq s,
\end{equation}
where
$$\eta_{\alpha~\beta}^{~J}:=\eta_{\alpha~\beta}^{~J~\a}.$$
Similar argument for $\theta_\gamma^{~k}\wedge\theta_k^{~\beta}$ in \eqref{dtheta2-2} yields
$$
\eta_{U~\beta}^{~\a~\gamma}=0~\text{   if   }~\b\ne \a$$
and hence \eqref{eq2} becomes
\begin{equation}\Label{eq5-2}
\Theta_U^{~\a} -\tilde\theta_U^{~\a} =\eta_{U}^{~\a\g}\phi_\g^{~\a}, \quad \a\leq s,
\end{equation}
where
$$\eta_{U}^{~\a\g}:=\eta_{U~\a}^{~\a~\g}.$$

Now if $J>n$ and $U>q-r$, then \eqref{dtheta2-1} and \eqref{dtheta2-2} become
\begin{equation}\Label{nonum1}
\eta_{\a~\b}^{~J}\left(\theta_\a^{~v}\wedge\theta_v^{~\b}+
\theta_\a^{~k}\wedge\theta_k^{~\beta}\right)=
\theta_\a^{~v}\wedge\Delta_v^{~J}+\theta_\alpha^{~k}\wedge\Omega_k^{~J} \mod \phi,
\end{equation}
i.e.\
\begin{equation}\Label{nonum1'}
\theta_\a^{~v}\wedge \left(\Delta_v^{~J}-\eta_{\a~\b}^{~J}\theta_v^{~\beta}
\right) + \theta_\alpha^{~k}\wedge \left(
\Omega_k^{~J}-\eta_{\a~\b}^{~J}\theta_k^{~\beta}
\right) =0\mod \phi,
\end{equation}
and
\begin{equation}\Label{nonum2}
\eta_{U}^{~\a\g}\left(\theta_\g^{~v}\wedge\theta_v^{~\a}+
\theta_\g^{~k}\wedge\theta_k^{~\a}\right)=\Omega_U^{~v} \wedge\theta_v^{~\a}+
\Delta_U^{~k}\wedge\theta_k^{~\a}
 \mod \phi,
\end{equation}
i.e.\
\begin{equation}\Label{nonum2'}
 \left(\Omega_U^{~v}-\eta_{U}^{~\a\g}\theta_\g^{~v}
\right)\wedge\theta_v^{~\a} +\left(\Delta_U^{~k}-\eta_{U}^{~\a\g}\theta_\g^{~k}
\right) \wedge\theta_k^{~\a}=0\mod \phi.
\end{equation}
Thus  using linear independence of $\theta_{\a}^{~k}$, $\theta_v^{~\a}$ and applying Cartan's Lemma, we obtain for each $\a$ the identities
\begin{align}
\Delta_v^{~J}&=\eta_{\a~\b}^{~J}\theta_v^{~\b}
\mod\{\phi, ~\theta_\alpha^{~+}, ~\theta_\a^{~-}\}, \quad J>n,
\Label{deltaJ}\\
\Omega_k^{~J}&=\eta_{\alpha~\beta}^{~J}\theta_k^{~\beta}\mod
\{\phi, ~\theta_\alpha^{~+}, ~\theta_\a^{~-}\}, \quad J>n, \Label{omegaJ}
\end{align}
and
\begin{align}
\Delta_U^{~k}&=\eta_{U}^{~\a\g}\theta_\g^{~k}
\mod\{\phi, ~\1{\theta_{\alpha}^{~+}}, ~\1{\theta_{\a}^{~-}}\}, \quad U>q-r,
\Label{deltaU}\\
\Omega_U^{~v}&=\eta_U^{~\a\g}\theta_\g^{~v}
\mod\{\phi, ~\1{\theta_{\alpha}^{~+}}, ~\1{\theta_{\a}^{~-}}\}, \quad U>q-r,
\Label{omegaU}
\end{align}
where $\theta_\a^{~+}$, $\theta_\a^{~-}$ are spans of $\theta_\a^{~k}$ and $\theta_\a^{~v}$ respectively.
Since $\Delta_v^{~J}$, $\Delta_U^{~k}$ are $(1,0)$ forms, $\theta_\a^{~v}$, $\theta_k^{~\a}$ are $(0,1)$, and $\Delta_v^{~J}$, $\Delta_U^{~k}$ and $\Omega_k^{~J}$, $\Omega_U^{~v}$ are independent of $\alpha$, we obtain
\begin{align}
\Delta_v^{~J}&=\eta_{\b}^{~J}\theta_v^{~\b}
\mod\{\phi, ~\theta_\alpha^{~+}  \}, \quad J>n,
\Label{deltaJ'}\\
\Omega_k^{~J}&=\eta_{\beta}^{~J}\theta_k^{~\beta}
\mod\{\phi, ~\theta_\alpha^{~+}, ~\theta_\a^{~-}\}, \quad J>n, \Label{omegaJ'}\\
\Delta_U^{~k}&=\eta_{U}^{~\g}\theta_\g^{~k}
\mod\{\phi,~\1{\theta_\a^{~-}} \}, \quad U>q-r,
\Label{deltaU'},\\
\Omega_U^{~v}&=\eta_U^{~\g}\theta_\g^{~v}
\mod\{\phi, ~\1{\theta_\alpha^{~+}}, ~\1{\theta_\a^{~-}}\}, \quad U>q-r.
\Label{omegaU'}
\end{align}
where
$$\eta_{\b}^{~J}=\eta_{\a~\b}^{~J}, \quad \eta_{U}^{~\g}=\eta_{U}^{\a~\g}$$
and \eqref{eq5-1} and \eqref{eq5-2} become
\begin{align}
\Theta_\alpha^{~J}&=\eta_{\beta}^{~J}\phi_\alpha^{~\beta}, \quad J>n,\Label{thetaJ2}\\
\Theta_U^{~\a}&=\eta_{U}^{~\g}\phi_\g^{~\a}, \quad U>q-r.\Label{thetaU2}
\end{align}

If on the other hand, $J=j\leq n$ and $U=u\leq q-r$, then \eqref{dtheta2-1} together with \eqref{ag} yields
\begin{equation}\Label{nonum3}
\theta_\a^{~v}\wedge \left(\Delta_v^{~j}-\delta_v^{~j}-\eta_{\a~\b}^{~j}\theta_v^{~\beta}
\right)=0\mod \phi, \theta^+,
\end{equation}
and \eqref{dtheta2-2} yields
\begin{equation}\Label{nonum4}
\left(\Delta_u^{~k}-\delta_u^{~k}-\eta_{u}^{~\a\g}\theta_\g^{~k}
\right) \wedge\theta_k^{~\a}=0\mod \phi,\theta^-.
\end{equation}
Then  using linear independence of $\theta_{\a}^{~k}$, $\theta_v^{~\a}$ and applying Cartan's Lemma, we obtain
$$\Delta_v^{~j}=\delta_v^{~j}+\eta_{\a~\b}^{~j}\theta_v^{~\b}\mod\{\phi, ~\theta^+\}$$
and
$$\Delta_u^{~k}=\delta_u^{~k}+\eta_{u}^{~\a\g}\theta_\g^{~k}\mod\{\phi, ~\theta_- \}.$$
Since $\Delta_v^{~j}$ is independent of $\alpha$, we obtain
\begin{equation}\Label{delta}
\Delta_v^{~j}=\delta_v^{~j}+\eta_{\b}^{~j}\theta_v^{~\b}+\eta_v^{~\g}\theta_\g^{~j}\mod \phi,
\end{equation}
where
$$\eta_{\b}^{~j}=\eta_{\a~\b}^{~j}, \quad \eta_{v}^{~\g}=\eta_{v}^{~\a\g}.$$
Hence \eqref{eq5-1} and \eqref{eq5-2} imply
\begin{align*}
\Theta_\alpha^{~j}&=\theta_\a^{~j}+\eta_{\beta}^{~j}\phi_\alpha^{~\beta}, \\
\Theta_u^{~\a}&=\theta_u^{~\a}+\eta_{u}^{~\g}\phi_\g^{~\a}.
\end{align*}
Then after applying change of frame of the source manifold given by
\begin{align*}
\widetilde Z_\a &=Z_\a,\\
\widetilde Z'_u &=\eta_u^{~\b}Z_\b+Z'_u,\\
\widetilde X_j &=X_j+C_j^{~\b}Z_\b,\\
\widetilde Y_\a &=Y_\a+A_\a^{~\b}Z_\b+V_\a^{~v}Z'_v+\eta_\a^{~j}X_j,
\end{align*}
where
$$\eta_v^{~\b}+V_v^{~\b}=0,$$
$$(A_\a^{~\b}+\overline{A_\b^{~\a}})-V_\a^{~v}V_v^{~\b}+\eta_\a^{~j}\eta_j^{~\b}=0,$$
$$C_j^{~\a}+\eta_j^{~\a}=0,$$
we can choose new $\theta_\a^{~j},~\theta_u^{~\a}$ and $\delta_u^{~j}$ such that
\begin{align}
\Theta_\a^{~j}&=\theta_\a^{~j},\Label{t-a}\\
\Theta_u^{~\a}&=\theta_u^{~\a}\Label{t-u},\\
\Delta_u^{~j}&=\delta_u^{~j}\quad\mod~\phi\Label{del}.
%\Label{delta-eq}
\end{align}

\begin{lemma}
Under the assumptions of Theorem~\ref{main}, we have
$$s=r.$$
\end{lemma}
\begin{proof}In Lemma~\ref{eds}, we showed that
$$
\Phi_a^{~b}=0,\quad a>s\quad \text{or}\quad b>s.$$
Suppose that $s<r$.
Choose a tangent vector $\xi$ transversal to $T^c$ such that  $\phi_r^{~r}(\xi)\neq 0$ and $\phi_\a^{~\b}(\xi)=0$ for $(\a,\b)\ne(r,r) $. Then $\Phi_{a}^{~b}(f_*(\xi))=0$ for all $a,b$. Therefore $f_*(\xi)\in {T'}^c$, which is a contradiction with the assumption of the theorem.

\end{proof}

\subsection{Determination of $\Theta_{\a}^{~J},~\Theta_U^{~\a}$ after a change of frame}
Using \eqref{thetaJ2} and \eqref{thetaU2}
and making a change of frame of the target manifold $M'$ given by
\begin{align*}
\widetilde Z'_U &=Z'_U + \eta_U^{~\b}Z_\b,~U>q-r\\
\widetilde X_J &=X_J+C_J^{~\b}Z_\b,~J>n\\
\widetilde Y_\a &=Y_\a+A_\a^{~\b}Z_\b+\sum_{V>q-r}H_\a^{~V}Z'_V+\sum_{J>n}\eta_\a^{~J}X_J,
\end{align*}
where
$$\eta_U^{~\b}+H_U^{~\b}=0, \quad H_U^{~\b}: = -\1{H_{\b}^{~U} },$$
$$(A_\a^{~\b}+\overline{A_\b^{~\a}}) - \sum_{U>q-r}H_\a^{~U}H_U^{~\b}+\sum_{J>n}\eta_\a^{~J}\eta_J^{~\b}=0,$$
$$C_J^{~\a}+\eta_J^{~\a}=0,$$
and fixing the remaining vectors of the frame, we can obtain new $\Theta_\a^{~J},~\Theta_U^{~\a}$ such that
\begin{align}
\Theta_\a^{~J}&=0,\\
\Theta_U^{~\a}&=0.
\end{align}

Summarizing we obtain
\begin{equation}\Label{summarizing}
\Phi_a^{~b}-\phi_a^{~b}=\Theta_a^J-\theta_a^J=\Theta_a^U-\theta_a^U=0
\end{equation}
and hence
% \eqref{eq1}, \eqref{eq2}, \eqref{dtheta1-1}, \eqref{dtheta
\begin{align*}
\Psi_a^{~\b}&=0 \mod ~\phi,~\quad a>r,\\
\Delta_v^{~J}&=0 \mod\{\phi, ~\theta_\alpha^{~+}  \}, \quad J>n,\\
\Omega_k^{~J}&=0 \mod\{\phi, ~\theta_\alpha^{~+}, ~\theta_\a^{~-}\}, \quad J>n, \\
\Delta_U^{~k}&=0 \mod\{\phi,~\1{\theta_{\a}^{~-}} \}, \quad U>q-r,\\
\Omega_U^{~v}&=0 
\mod\{
\phi, 
~\1{\theta_{\alpha}^{~+}}, 
~\1{\theta_\a^{~-}}
\}, \quad U>q-r.
\end{align*}

\subsection{Determination of $\Delta$ and $\Omega$ modulo $\phi$.}
If $r\ge 2$, i.e.\ $\a$ admits at least two values, then as in \cite{KZ}, we conclude that the right-hand sides in the previous equations are in fact independent of $\a$. That is,
\begin{equation}
\Delta_u^{~J}=\Omega_k^{~J}= \Delta_U^{~k}=\Omega_U^{~u}=0\quad\mod\phi.
\end{equation}

Suppose now $r=1$, i.e.\ $\a=1$. We will analyze the Gauss equations 
%obtained by differentiating
%\begin{align}
%\hat\delta_j^{~k}\left(\Psi_\a^{~\a}-\psi_\a^{~\a}\right)-\left(\Omega_j^{~k}-\omega_j^{~k}\right),\\
%\hat\delta_u^{~v}\left(\Psi_\a^{~\a}-\psi_\a^{~\a}\right)-\left(\Omega_u^{~v}-\omega_u^{~v}\right).
%\end{align}
as in \cite{W79}.
From above equations and \eqref{del} we obtain
\begin{align}
\Delta_u^{~k}&=\delta_u^{~k}+\eta_u^{~k}\phi_1^{~1},\Label{Delta-delta}\\
\Delta_u^{~J}&=A_{u~k}^{~J}\theta_1^{~k} \mod\phi, \quad J>n,			\Label{Delta-J}\\
\Omega_k^{~J}&=\2A_{v~k}^{~J}\theta_1^{~v}
	+B_{k~j}^{~J}\theta_1^{~j} 
	\mod\phi,  \quad J>n, 
	\Label{Om}\\
\Delta_U^{~k}&=A_U^{~kv}\theta_v^{~1} \mod\phi, \quad U>q-r,
	\Label{Uk}\\
\Omega_U^{~u}&
	=\2A_U^{~kv}\theta_k^{~1}
	+B_U^{~uv}\theta_v^{~1}
	\mod\phi, \quad U>q-r.
	\Label{Uu}
\end{align}
Substituting into \eqref{nonum1'} and \eqref{nonum2'} (with $\eta_{\a~\b}^{~J}=\eta_U^{~\a\b}=0$)
we obtain
$$
\2A_{v~k}^{~J}=A_{v~k}^{~J}, \quad
\2A_U^{~kv} = A_U^{~kv},\quad
B_{k~l}^{~J}=B_{l~k}^{~J},\quad B_U^{~uv}=B_U^{~vu}.
$$
Consider the structure equations $d\pi =\pi\wedge \pi$
obtained by differentiating the following identities from \eqref{summarizing}: 
$$\Theta_\a^{~j}=\theta_\a^{~j},~\quad \Theta_\a^{~u}=\theta_\a^{~u}.$$
Then we obtain (with $\a=1$):
\begin{align}
\Psi_\a^{~\a}\wedge\theta_\a^{~j}+\theta_\a^{~v}\wedge\Delta_v^{~j}+
	\theta_\a^{~k}\wedge\Omega_k^{~j}+\phi_\a^{~\a}\wedge\Sigma_\a^{~j}
=&\psi_\a^{~\a}\wedge\theta_\a^{~j}+\theta_\a^{~v}\wedge\delta_v^{~j}+\theta_\a^{~k}\wedge\omega_k^{~j}+
\phi_\a^{~\a}\wedge\sigma_\a^{~j},\label{EQ1}\\
\Psi_\a^{~\a}\wedge\theta_\a^{~u}
	+\theta_\a^{~v}\wedge\Omega_v^{~u}+
	\theta_\a^{~k}\wedge\Delta_k^{~u}
	+\phi_\a^{~\a}\wedge\Sigma_\a^{~u}
=&\psi_\a^{~\a}\wedge\theta_\a^{~u}+\theta_\a^{~v}\wedge\omega_v^{~u}+\theta_\a^{~k}\wedge\delta_k^{~u}+
\phi_\a^{~\a}\wedge\sigma_\a^{~u},\label{EQ2}
\end{align}
which yield using \eqref{Delta-delta}:
\begin{align}
\left[\hat\delta_k^{~j}\left(\Psi_\a^{~\a}-\psi_\a^{~\a}\right)-\left(\Omega_k^{~j}-\omega_k^{~j}\right)\right]\wedge\theta_\a^{~k}
+\phi_\a^{~\a}\wedge\left(\Sigma_\a^{~j}-\sigma_\a^{~j}-\eta_v^{~j}\theta_\a^{~v}\right)&=0,\Label{dtheta1}\\
\left[\hat\delta_v^{~u}\left(\Psi_\a^{~\a}-\psi_\a^{~\a}\right)-\left(\Omega_v^{~u}-\omega_v^{~u}\right)\right]\wedge\theta_\a^{~v}
+\phi_\a^{~\a}\wedge\left(\Sigma_\a^{~u}-\sigma_\a^{~u}-\eta_k^{~u}\theta_\a^{~k}\right)
&=0,\Label{dtheta2}
\end{align}
where
$$\eta_k^{~u}:=-\overline{\eta_u^{~k}}$$
and $\hat\delta$ denotes the Kronecker delta.
Then
by Cartan's lemma applied to \eqref{dtheta1}, we obtain
$$
\hat\delta_k^{~j}
	\left(\Psi_\a^{~\a}-\psi_\a^{~\a}\right)
	-\left(\Omega_k^{~j}-\omega_k^{~j}\right)
	=0\quad
	\mod\theta_\a^{~+},\phi,
$$
and by \eqref{dtheta2}, we obtain
$$\hat\delta_v^{~u}
	\left(\Psi_\a^{~\a}-\psi_\a^{~\a}\right)
	-\left(\Omega_v^{~u}-\omega_v^{~u}\right)
	=0\quad
	\mod\theta_\a^{~-},\phi,$$
which imply
\begin{equation}
\Psi_\a^{~\a}-\psi_\a^{~\a}
	=\Omega_j^{~j}-\omega_j^{~j}
	=\Omega_u^{~u}-\omega_u^{~u}\quad
	\mod\theta_\a^{~+},\theta_\a^{~-},\phi.
\label{psi-omega}
\end{equation}
Using symmetry relation for $\Omega$ and the fact that $\theta_\a^{~+}$, $\theta_\a^{~-}$ are $(1,0)$, $(0,1)$ respectively , we obtain
\begin{align*}
\Omega_k^{~j}&=\omega_k^{~j}\quad\mod\phi,\quad j\neq k,\\
\Omega_v^{~u}&=\omega_v^{~u}\quad\mod\phi,\quad u\neq v,
\end{align*}
and
\begin{equation}\Label{p-o}
\Psi_\a^{~\a}-\psi_\a^{~\a}=\Omega_j^{~j}-\omega_j^{~j}=\Omega_u^{~u}-\omega_u^{~u}\quad\mod\phi.
\end{equation}

Now consider the structure equation obtained by differentiating
the following identity from \eqref{summarizing}:
$$\Phi_\a^{~\a}-\phi_\a^{~\a}=0,$$
which yields
\begin{equation}\label{d-phi}
\left(\Psi_\a^{~\a}-\psi_\a^{~\a}-\hat\Psi_\a^{~\a}+\hat\psi_\a^{~\a}\right)\wedge\phi_\a^{~\a}=0,
\end{equation}
or equivalently
$$\left(\Psi_\a^{~\a}-\psi_\a^{~\a}+\Psi_{\bar\a}^{~\bar\a}-\psi_{\bar\a}^{~\bar\a}\right)\wedge\phi_\a^{~\a}=0.$$
Let
$$\Psi_\a^{~\a}-\psi_\a^{~\a}=\hat\Psi_{\a}^{~\a}-\hat\psi_{\a}^{~\a}+g\phi_\a^{~\a}$$
for some pure imaginary function $g$.
Applying a real vector change (see Defintion~\ref{changes}) of the source manifold defined by
$$\widetilde Y_\a=Y_\a+\frac g2  Z_\a,$$
we may assume that
\begin{equation}\Label{this}
\Psi_\a^{~\a}-\psi_\a^{~\a}=\hat\Psi_{\a}^{~\a}-\hat\psi_{\a}^{~\a}.
\end{equation}

By \eqref{dtheta1}, \eqref{dtheta2}, we obtain
\begin{align}
\Sigma_\a^{~j}-\sigma_\a^{~j}=g_k^{~j}\theta_\a^{~k}+\eta_v^{~j}\theta_\a^{~v}\quad\mod\phi,\Label{sigma1}\\
\Sigma_\a^{~u}-\sigma_\a^{~u}=g_v^{~u}\theta_\a^{~v}
	+\eta_k^{~u}\theta_\a^{~k}\quad\mod\phi,\Label{sigma2}
\end{align}
for suitable functions $g_k^{~j}$, $g_v^{~u}$.
Then, using
the structure identities $d\pi=\pi\wedge\pi$ obtained from differentiating \eqref{this}, we obtain
$$\theta_\a^{~v}\wedge(\Sigma_u^{~\a}-\sigma_u^{~\a})+\theta_\a^{~k}\wedge(\Sigma_k^{~\a}-\sigma_k^{~\a})=
(\Sigma_\a^{~v}-\sigma_\a^{~v})\wedge\theta_v^{~\a}+
(\Sigma_\a^{~k}-\sigma_\a^{~k})\wedge\theta_k^{~\a}\quad\mod\phi,$$
which implies
$$g_k^{~j}=\overline{g_j^{~k}}\quad g_u^{~v}=\overline{g_v^{~u}}.$$

Then substituting \eqref{sigma1}, \eqref{sigma2} into  \eqref{dtheta1} and \eqref{dtheta2} imply
\begin{align}
\hat\delta_k^{~j}\left(\Psi_\a^{~\a}-\psi_\a^{~\a}\right)-\left(\Omega_k^{~j}-\omega_k^{~j}\right)+g_j^{~k}\phi_\a^{~\a}&
=0,\label{Omega-omega1}\\
\hat\delta_v^{~u}\left(\Psi_\a^{~\a}-\psi_\a^{~\a}\right)-\left(\Omega_v^{~u}-\omega_v^{~u}\right)+g_v^{~u}\phi_\a^{~\a}&
=0.\label{Omega-omega2}
\end{align}

Now differentiate \eqref{Omega-omega1} and use the structure equations $d\pi=\pi\wedge\pi$ together with \eqref{Delta-delta}, \eqref{p-o} to obtain
\begin{align*}
&\sum_{V>q-r}\Delta_k^{~V}\wedge\Delta_V^{~j}+\sum_{K>n}\Omega_k^{~K}\wedge\Omega_K^{~j}\\
=~&\hat\delta_k^{~j}\left(\theta_\a^{~v}\wedge\left(\Sigma_v^{~\a}-\sigma_v^{~\a}\right)
+\theta_\a^{~l}\wedge\left(\Sigma_l^{~\a}-\sigma_l^{~\a}\right)\right)
+\hat\delta_l^{~j}\left(\theta_\a^{~l}\wedge\left(\Sigma_k^{~\a}-\sigma_k^{~\a}\right)
-\theta_k^{~\a}\wedge(\Sigma_\a^{~l}-\sigma_\a^{~l})\right)\\
	+~&g_k^{~j}\left(\theta_\a^{~v}\wedge\theta_v^{~\a}
	+\theta_\a^{~l}\wedge\theta_l^{~\a}\right) \mod \phi.
\end{align*}
Substituting the identities following \eqref{Delta-delta}
as well as \eqref{sigma1}, \eqref{sigma2},
we now obtain
\begin{align}
B_{k~l}^{~K} B_{~K}^{j~m}&=g_l^{~m}\hat\delta_k^{~j}+g_l^{~j}\hat\delta_k^{~m}+g_k^{~m}\hat\delta_l^{~j}+g_k^{~j}\hat\delta_l^{~m},
\Label{2nd-j1}\\
A_{u~k}^{~K}B_{~K}^{j~m}&=\eta_u^{~m}\hat\delta_k^{~j}+\eta_u^{~j}\hat\delta_k^{~m},\Label{2nd-j2}\\
A_{ku}^{~V}A_{V}^{~jv}
+A_{u~k}^{~K}A_{~K}^{v~j}&=g_u^{~v}\hat\delta_k^{~j}+g_k^{~j}\hat\delta_u^{~v}.\Label{2nd-u}
\end{align}

Now Lemma~5.3 from \cite{EHZ05} implies that $B_{k~l}^{~K}=0$
provided $n'<2n$, which is part of our assumptions (recall that $n=p-r$).
Therefore putting all indices to be $j$ in the right-hand side of \eqref{2nd-j1} we obtain $g_j^{~j}=0$ and hence $g_j^{~k}=0$ since it is hermitian. Then \eqref{2nd-j2} for $m=k=j$ yields 
$$\eta_u^{~j}=0.$$
Then \eqref{Om} reads
$$\Omega_k^{~J}=A_{u~k}^{~J}\theta_\a^{~u}\mod\phi, 
	\quad J>n.$$
Then by differentiation and the structure identities $d\pi=\pi\wedge\pi$ we obtain
$$\Delta_k^{~V}\wedge\Delta_V^{~J}+\Omega_k^{~L}\wedge\Omega_L^{~J}+\theta_k^{~\a}\wedge\Sigma_\a^{~J}
	=A_{u~k}^{~J}\theta_\a^{~l}\wedge
	\delta_l^{~u}\quad
	\mod \theta^-,\phi.$$
By substituting \eqref{Delta-delta} and \eqref{Delta-J},
we obtain
$$A_{u~l}^{~J}\delta_k^{~u}\wedge\theta_1^{~l}
	=A_{u~k}^{~J}\theta_1^{~l}\wedge\delta_l^{~u}.$$
Hence
$$A_{u~k}^{~J}=0$$
and \eqref{2nd-u} becomes
$$\sum_{V>q-r} A_{l u}^{~V} A_{V}^{~j w}=g_u^{~w}\hat\delta_l^{~j}.$$
Therefore for each fixed $u$, the vectors $A^{ju}:=(A_{V}^{~ju}) \in \C^{q'-r'-(q-r)}$, $j=1,\ldots,n,$ are orthogonal to each other.
If
$$g_u^{~u}\neq 0,$$
then
$$A_{~j u}^{V} A_{V}^{~j u}=g_u^{~u}\neq 0,\quad j=1,\ldots,n,$$
which leads to a contradiction, since we assumed $q'-r'<n$.
Consequently  $A_{~j u}^{V}=0$,  $g_u^{~u}=0$ and hence $g_u^{~v}=0$, and \eqref{Uk}, \eqref{Uu} yield
$$\Delta_U^{~j}=0\quad\mod\phi, \quad U>q-r$$
and
$$\Omega_U^{~v}=B_U^{~vw}\theta_w^{~\a}\quad\mod\phi,  \quad U>q-r.$$
By differentiating the last equation and following the same argument as before for $\Omega_k^{~J}$, we obtain
$$B_U^{~vw}=0.$$

%Suppose $n=2$, then by assumption, we have $q'-r'=1<2=n$. Hence $q'-r'=q-r=1$ and \eqref{2nd-u} becomes
%$$A_{u~k}^{~K}A_{~K}^{v~j}=g_u^{~v}\hat\delta_k^{~j}+g_k^{~j}\hat\delta_u^{~v},$$
%which together with \eqref{2nd-j1} and \eqref{2nd-j2} implies
%$$B=A=0.$$
%

Summing up, we obtain
\begin{lemma}\Label{summ}
\begin{align*}
\Delta_u^{~j}&=\delta_u^{~j},\\
\Psi_a^{~\b}&=0 \mod ~\phi,~\quad a>r,\\
\Delta_v^{~J}&=\Omega_k^{~J}=0 \mod\phi\, \quad J>n,\\
\Delta_U^{~k}&=\Omega_U^{~v}=0  \mod\phi, \quad U>q-r.
\end{align*}
\end{lemma}

%%%%%%%%%%%%%%%%%%%%%%%%%%%%%%%%%%%%%%%%%%%%%%%%%%%%%%%%%%%%%%%%%%%%%%%%%%%%%%%%%%%%%%%%%%%%%%%%
\section{Determination of the second fundamental forms}
Next, we shall determine all second fundamental forms
$$\Psi_a^{~\b},~\quad a>r,$$
$$\Delta_v^{~J},~\Omega_k^{~J},~ \Sigma_\b^{~J}\quad J>n,$$
$$\Delta_U^{~k},~\Omega_U^{~v} ,~\Sigma_U^{~\b}\quad U>q-r.$$

\subsection{Determination of $\Psi_{a}^{~\b}$ for $a>r$}
In view of Lemma~\ref{summ} we can write
\begin{equation}\label{psi-a}
\Psi_a^{~\beta}=h_{a~\delta}^{~\b~\g}\phi_\g^{~\delta}, \quad a>r.
\end{equation}
Since
$$\Phi_a^{~b}=\Theta_a^{~J}=\Theta_a^{~U}=0,\quad a>r,$$
differentiation of \eqref{psi-a} and using the structure equations $d\pi=\pi\wedge\pi$ yields
\begin{equation*}
h_{a~\delta}^{~\b~\g}\left(\theta_\g^{~u}\wedge\theta_u^{~\delta}+\theta_\g^{~k}\wedge\theta_k^{~\delta}\right)=0\quad\mod\phi,
\end{equation*}
which implies
$$h_{a~\delta}^{~\b~\g}=0,$$
and hence
\begin{equation}
\Psi_a^{~\b}=0, \quad a>r.
\end{equation}

\subsection{Determination of $\Delta_u^{~J},~\Omega_{k}^{~J}$ for $J>n$}
In view of Lemma~\ref{summ} we write
$$\Delta_u^{~J}=h_{u~\b}^{~J~\a}\phi_\a^{~\b}, \quad J>n.$$
By differentiation and using structure identities as before, we obtain
$$\theta_u^{~\b}\wedge\Sigma_\b^{~J}=h_{u~\b}^{~J~\a}\left(\theta_\a^{~v}\wedge\theta_v^{~\b}+\theta_\a^{~k}\wedge\theta_k^{~\b}\right)
\quad\mod\phi.$$
Since the left hand side contains no $\theta_\a^{~k}\wedge\theta_k^{~\b}$ terms, we obtain
$$h_{u~\b}^{~J~\a}=0,$$
i.e.,
$$\Delta_u^{~J}=0$$
and
$$\Sigma_\b^{~J}=0\quad\mod\theta_u^{~\b},\phi.$$
Similar argument implies
$$\Omega_k^{~J}=0$$
and
$$\Sigma_\b^{~J}=0\quad\mod\theta_k^{~\b},\phi.$$
Consequently, we obtain
$$\Sigma_\b^{~J}=0\quad\mod\phi.$$

\subsection{Determination of $\Delta_U^{~j},~\Omega_{U}^{~v}$ for $U>q-r$}
Let
$$\Delta_U^{~j}=h_{U~\b}^{~j~\a}\phi_\a^{~\b}.$$
By differentiation and structure identities as before, we obtain
$$\Sigma_U^{~\b}\wedge\theta_\b^{~j}=h_{U~\b}^{~j~\a}\left(\theta_\a^{~v}\wedge\theta_v^{~\b}+\theta_\a^{~k}\wedge\theta_k^{~\b}\right)
\quad\mod\phi.$$
Since the left hand side contains no $\theta_\a^{~v}\wedge\theta_v^{~\b}$ terms, we obtain
$$h_{U~\b}^{~j~\a}=0,$$
i.e.,
$$\Delta_U^{~j}=0$$
and
$$\Sigma_U^{~\b}=0\quad\mod\theta_\b^{~j},\phi.$$
Similar argument implies
$$\Omega_U^{~v}=0$$
and
$$\Sigma_U^{~\b}=0\quad\mod\theta_\b^{~v},\phi.$$
Consequently, we obtain
$$\Sigma_U^{~\b}=0\quad\mod\phi.$$

\subsection{Determination of $\Sigma_{\b}^{~J}$ and $\Sigma_{U}^{~\b}$ for $U>q-r$ and $J>n$}
Now let
$$\Sigma_\g^{~J}=h_{\g~\b}^{~J~\a}\phi_\a^{~\b}.$$
By differentiation, we obtain
$$0=
h_{\g~\b}^{~J~\a}\left(\theta_\a^{~v}\wedge\theta_v^{~\b}+\theta_\a^{~k}\wedge\theta_k^{~\b}\right)
\quad\mod\phi,$$
which implies
$$h_{\g~\b}^{~J~\a}=0,$$
i.e.,
$$\Sigma_\g^{~J}=0.$$
Similar argument implies
$$\Sigma_U^{~\b}=0.$$

\medskip
We summarize the obtained alignment of the connection forms:

\begin{proposition}\Label{eds-2}
For any local CR embedding $f$ from $S_{p,q,r}$ into $S_{p',q',r'}$
satisfying the assumptions of either Theorem~\ref{main},
there is a choice of sections of the frame bundles
$\6B_{p,q}\to S_{p,q}$ and $\6B_{p',q'}\to S_{p',q'}$
such that
\begin{align}
\Phi_a^{~b}&=\phi_a^{~b}, \quad \Theta_U^{~a}=\theta_U^{~a},\quad
\Theta_a^{~J}=\theta_a^{~J}, \quad \Delta_u^{~j}=\delta_u^{~j},\\
%\Psi_\a^{~\b}&=\psi_{\a}^{~\b} \mod \phi,~\quad \a\ne\b,\Label{psi-eqn2}\\
\Psi_a^{~\beta}&=\Delta_u^{~K}=\Delta_U^{~j}=
%\Psi_\alpha^{~\alpha}-\psi_\alpha^{~\alpha}&=\Omega_j^{~j}-\omega_j^{~j} \mod \phi,\Label{psi-eqn2}\\
%\Omega_j^{~k}&=\omega_j^{~k} \mod\phi,\quad j\neq k,\Label{omega-eq1}\\
\Omega_j^{~K}=\Omega_U^{~v}=
\Sigma_{\a}^{~K}=\Sigma_U^{~\b}=
0, \quad a>r, ~U>q-r,~ K>n.
\end{align}
\end{proposition}
\medskip

%\begin{remark}
%Similarly to Remark~\ref{back1},
%we can restrict to changing only the section of the second bundle
%$\6B_{p',q'}\to S_{p',q'}$.
%\end{remark}

\section{Splitting of the image with suitable dimensions}
We shall write $Gr(V,s)$ for the Grassmanian of all $s$-dimensional subspaces of $V$.
\begin{proposition}\label{flatness}
Under the assumptions of Theorem~\ref{main}, there exist
vector subspaces $V_0,V_1,V_2\subset \C^{p'+q'}$
of dimension 
$$\dim V_0 = p+q, \quad \dim V_1 = r'-r, \quad \dim V_2 = n'+q'-r' - (n+q-r)$$
 that form a direct sum, such that the basic form $\langle\cdot,\cdot\rangle$ is null when restricted to $V_{1}$, nondegenerate of signature $(p,q)$ when restricted to $V_{0}$, and nondegenerate of signature 
 $(n'-n,q'-r'-q+r)$ when restricted to $V_{3}$,
and such that whenever $x\in S_{p,q,r}$ and $f(x)$ is defined, we have
\begin{equation}\Label{6.1}
f(x) = W_0\oplus V_1\oplus W_2 \in Gr(V_0,q) \oplus V_1 \oplus  Gr(V_2, (q'-r')-(q-r)),
\end{equation}
such that the basic form restricted to $W_0$ has rank $r$.
\end{proposition}

\begin{proof}
Denote by $M\subset S_{p,q,r}$ the open subset where $f$ is defined.
Let $Z, Z', X, Y$ be collections of constant vector fields valued in $\mathbb{C}^{p'+q'}$ as in Section~\ref{adapted},  forming a $S_{p',q',r'}$-frame adapted to $f(M)$ at a fixed reference point in $f(M)$. Let
\begin{align}
\widetilde Z_a&
	=\lambda_a^{~b}Z_b
	+\mu_a^{~V}Z'_V+\eta_a^{~K}X_K+\zeta_a^{~b}Y_b, \Label{adapted1}\\
\widetilde Z'_U&
	=\lambda_U^{~b}Z_b+\mu_U^{~V}Z'_V+\eta_U^{~K}X_K+\zeta_U^{~b}Y_b, \Label{adapted1'}\\
\widetilde X_J&=\lambda_J^{~b}Z_b+\mu_J^{~V}Z'_V+\eta_J^{~K}X_K+\zeta_J^{~b}Y_b,\Label{adapted2}\\
\widetilde Y_a&=\hat\lambda_a^{~b}Z_b+\hat\mu_a^{~V}Z'_V+\hat\eta_a^{~K}X_K+\hat\zeta_a^{~b}Y_b\Label{adapted3}
\end{align}
be an adapted $S_{p',q',r'}$-frame along $f(M)$.
Set
\begin{equation}\Label{A}
A
:=
\begin{pmatrix}
\lambda_a^{~b} & \mu_a^{~V} & \eta_a^{~K} & \zeta_a^{~b}\\
\lambda_U^{~b} & \mu_U^{~V} & \eta_U^{~K} & \zeta_U^{~b}\\
\lambda_J^{~b} & \mu_J^{~V} & \eta_J^{~K} & \zeta_J^{~b}\\
\hat\lambda_a^{~b} & \hat\mu_a^{~V} & \hat\eta_a^{~K} & \hat\zeta_a^{~b}\\
\end{pmatrix},
\end{equation}
so that \eqref{adapted1} - \eqref{adapted3} take the form
\begin{equation}\Label{A-eq}
\begin{pmatrix}
\2Z_a\\\2Z'_U\\ \2X_J\\ \2Y_a
\end{pmatrix}
=
A
\begin{pmatrix}
Z_b\\Z'_V\\ X_K\\ Y_b
\end{pmatrix}
.
\end{equation}
Since $Z, Z', X, Y$ form an adapted frame at a reference point of $f(M)$, we may assume that
\begin{equation}\Label{initial}
A=I_{p'+q'}
\end{equation}
at the reference point.
Since $Z, Z', X, Y$ are constant vector fields, i.e., $dZ=dZ'=dX=dY=0$,
differentiating \eqref{A-eq} and using \eqref{pi}
for $\2Z, \2Z', \2X, \2Y$
we obtain
\begin{equation}\Label{dA}
dA=\Pi A,
\end{equation}
where $\Pi$ is the connection matrix of $S_{p',q',r'}$, i.e.\ we have
\begin{equation}\Label{dA-expand}
dA=
\begin{pmatrix}
\Psi_{a}^{~b} & \Theta_a^{~V} & \Theta_{a}^{~K} & \Phi_{a}^{~b}\\
\Sigma_U^{~b} & \Omega_U^{~V} & \Delta_U^{~K} & \Theta_U^{~b}\\
\Sigma_{J}^{~b} & \Omega_J^{~V} & \Omega_{J}^{~K} & \Theta_{J}^{~b}\\
\Xi_{a}^{~b} & \Sigma_a^{~V} & \Sigma_{a}^{~K} & \3\Psi_{a}^{~b}\\
\end{pmatrix}
A.
\end{equation}

Next, it follows from Proposition~\ref{eds-2} that
\begin{equation}\Label{dZa}
d\2Z_a= \sum_{b>r}\Psi_{a}^{~b} \2Z_{b}, \quad a>r,
\end{equation}
in particular,
the span of $\2Z_{a}$, $a>r$,
is independent of the point in $f(M)$. Therefore
together with \eqref{adapted1} and \eqref{initial}, we conclude
\begin{equation}\Label{Z_a-determined}
\mu_a^{~V}=\eta_a^{~K}=\zeta_a^{~b}=0, \quad a>r.
\end{equation}
Then using again Proposition~\ref{eds-2},
\begin{align}
d\2Z_U &=\sum_{b>r}\Sigma_U^{~b}Z_b+\sum_{V>q-r}\Omega_U^{~V}Z_V+\sum_{K>n}\Delta_U^{~K}X_K,\quad U>q-r,\Label{dZU}\\
d\2 X_J &=\sum_{b>r}\Sigma_J^{~b}Z_b+\sum_{V>q-r}
		\Delta_J^{~V}Z_V+\sum_{K>n}\Omega_J^{~K}X_K,\quad J>n, \Label{dXJ},
\end{align}
which together with \eqref{dZa} imply that the span of $\2 Z_a,\2Z'_U,\2X_J$ is the same as the span of
$Z_a,~Z'_U,X_J$ where $a>r, ~U>q-r,~J>n$.
Let
$$V_1:
=\sp\{Z_a : a>r\}
=\sp\{\2Z_a : a>r\}
.$$
Consider hermitian form $\langle\cdot,\cdot\rangle$ as in \eqref{form}.
By definition of adapted frame, $\langle\cdot,\cdot\rangle$
 restricted to $V_1$ is null.
Choose $V_2$ transversal to $V_1$ such that
$$V_2
=\sp\{Z'_U, X_J : U>q-r, ~J>n\}
%=\sp\{\2Z_a, \2Z'_U, \2X_J : a>r, ~U>q-r, ~J>n\}
.$$
Then $V_1$ is the kernel of $\langle\cdot,\cdot\rangle|_{V_1\oplus V_2}$ and 
 $\langle\cdot,\cdot\rangle$ restricted to $V_2$ is nondegenerate with $q'-r'-(q-r)$ negative and $n'-n$ positive eigenvalues.
Furthermore, \eqref{dA-expand} implies
\begin{equation}
	\begin{pmatrix}
d\mu_{a}^{~V} \\ d \mu_U^{~V} \\ d\mu_{J}^{~V} \\ d\hat\mu_{a}^{~V}
\end{pmatrix}
=
	\begin{pmatrix}
\Psi_{a}^{~b} & \Theta_a^{~W} & \Theta_{a}^{~L} & \Phi_{a}^{~b}\\
\Sigma_U^{~b} & \Omega_U^{~W} & \Delta_U^{~L} & \Theta_U^{~b}\\
\Sigma_{J}^{~b} & \Omega_J^{~W} & \Omega_{J}^{~L} & \Theta_{J}^{~b}\\
\Xi_{a}^{~b} & \Sigma_a^{~W} & \Sigma_{a}^{~L} & \3\Psi_{a}^{~b}\\
\end{pmatrix}
	\begin{pmatrix}
\mu_{b}^{~V} \\ 
\mu_W^{~V} \\ 
\mu_{L}^{~V} \\ 
\hat\mu_{b}^{~V}
\end{pmatrix},
\end{equation}

\begin{equation}
	\begin{pmatrix}
d\eta_{a}^{~K} \\ d\eta_U^{~K} \\ d\eta_{J}^{~K} \\ d\hat\eta_{a}^{~K}
\end{pmatrix}
=
	\begin{pmatrix}
\Psi_{a}^{~b} & \Theta_a^{~W} & \Theta_{a}^{~L} & \Phi_{a}^{~b}\\
\Sigma_U^{~b} & \Omega_U^{~W} & \Delta_U^{~L} & \Theta_U^{~b}\\
\Sigma_{J}^{~b} & \Omega_J^{~W} & \Omega_{J}^{~L} & \Theta_{J}^{~b}\\
\Xi_{a}^{~b} & \Sigma_a^{~W} & \Sigma_{a}^{~L} & \3\Psi_{a}^{~b}\\
\end{pmatrix}
	\begin{pmatrix}
\eta_{b}^{~K} \\ \eta_W^{~K} \\ \eta_{L}^{~K} \\ \hat\eta_{b}^{~K}
\end{pmatrix}
\end{equation}
and
\begin{equation}
	\begin{pmatrix}
d\zeta_{a}^{~c} \\ 
d\zeta_U^{~c} \\ d\zeta_{J}^{~c} \\ d\hat\zeta_{a}^{~c}
\end{pmatrix}
=
	\begin{pmatrix}
\Psi_{a}^{~b} & \Theta_a^{~W} & \Theta_{a}^{~L} & \Phi_{a}^{~b}\\
\Sigma_U^{~b} & \Omega_U^{~W} & \Delta_U^{~L} & \Theta_U^{~b}\\
\Sigma_{J}^{~b} & \Omega_J^{~W} & \Omega_{J}^{~L} & \Theta_{J}^{~b}\\
\Xi_{a}^{~b} & \Sigma_a^{~W} & \Sigma_{a}^{~L} & \3\Psi_{a}^{~b}\\
\end{pmatrix}
	\begin{pmatrix}
\zeta_{b}^{~c} \\ \zeta_W^{~c} \\ \zeta_{L}^{~c} \\ \hat\zeta_{b}^{~c}
\end{pmatrix}.
\end{equation}
Now restricting to $a=\a\le r$, $U=u \le q-r$ and $J=j\le n$,
$V>q-r$, $K>n$, $c>r$,
and using \eqref{Z_a-determined} and Proposition~\ref{eds-2} taken into account,  we obtain
\begin{equation}
	\begin{pmatrix}
d\mu_{\a}^{~V} \\ d \mu_u^{~V} \\ d\mu_{j}^{~V} \\ d\hat\mu_{\a}^{~V}
\end{pmatrix}
=
	\begin{pmatrix}
\Psi_{\a}^{~\b} & \Theta_\a^{~w} & \Theta_{\a}^{~l} & \Phi_{\a}^{~\b}\\
\Sigma_u^{~\b} & \Omega_u^{~w} & \Delta_u^{~l} & \Theta_u^{~\b}\\
\Sigma_{j}^{~\b} & \Omega_j^{~w} & \Omega_{j}^{~l} & \Theta_{j}^{~\b}\\
\Xi_{\a}^{~\b} & \Sigma_\a^{~w} & \Sigma_{\a}^{~l} & \3\Psi_{\a}^{~\b}\\
\end{pmatrix}
	\begin{pmatrix}
\mu_{\b}^{~V} \\ 
\mu_w^{~V} \\ 
\mu_{k}^{~V} \\ 
\hat\mu_{\b}^{~V}
\end{pmatrix},
\end{equation}

\begin{equation}
	\begin{pmatrix}
d\eta_{\a}^{~K} \\ d\eta_u^{~K} \\ d\eta_{j}^{~K} \\ d\hat\eta_{\a}^{~K}
\end{pmatrix}
=
	\begin{pmatrix}
\Psi_{\a}^{~\b} & \Theta_\a^{~w} & \Theta_{\a}^{~l} & \Phi_{\a}^{~\b}\\
\Sigma_u^{~\b} & \Omega_u^{~w} & \Delta_u^{~l} & \Theta_u^{~\b}\\
\Sigma_{j}^{~\b} & \Omega_j^{~w} & \Omega_{j}^{~l} & \Theta_{j}^{~\b}\\
\Xi_{\a}^{~\b} & \Sigma_\a^{~w} & \Sigma_{\a}^{~l} & \3\Psi_{\a}^{~\b}\\
\end{pmatrix}
	\begin{pmatrix}
\eta_{\b}^{~K} \\ \eta_w^{~K} \\ \eta_{l}^{~K} \\ \hat\eta_{\b}^{~K}
\end{pmatrix}
\end{equation}
and
\begin{equation}
	\begin{pmatrix}
d\zeta_{\a}^{~c} \\ 
d\zeta_u^{~c} \\ d\zeta_{j}^{~c} \\ d\hat\zeta_{\a}^{~c}
\end{pmatrix}
=
	\begin{pmatrix}
\Psi_{\a}^{~\b} & \Theta_\a^{~w} & \Theta_{\a}^{~l} & \Phi_{\a}^{~\b}\\
\Sigma_u^{~\b} & \Omega_u^{~w} & \Delta_u^{~l} & \Theta_u^{~\b}\\
\Sigma_{j}^{~\b} & \Omega_j^{~w} & \Omega_{j}^{~l} & \Theta_{j}^{~\b}\\
\Xi_{\a}^{~\b} & \Sigma_\a^{~w} & \Sigma_{\a}^{~l} & \3\Psi_{\a}^{~\b}\\
\end{pmatrix}
	\begin{pmatrix}
\zeta_{\b}^{~c} \\ \zeta_w^{~c} \\ \zeta_{l}^{~c} \\ \hat\zeta_{\b}^{~c}
\end{pmatrix}.
\end{equation}
Thus each of the vector valued functions $\mu^V=(\mu_\a^{~V}, \mu_u^{~V}, \mu_j^{~V}, \hat\mu_\b^{~V})$,
$\eta^K:=(\eta_\alpha^{~K}, \eta_u^{~K}, \eta_j^{~K}, \hat\eta_\b^{~K})$ and
$\zeta^c:=(\zeta_\alpha^{~c}, \zeta_u^{~c}, \zeta_j^{~c}, \hat\zeta_\b^{~c})$ for
fixed $V>q-r$, $K>n$ and $c>r$
satisfies a complete system of first order  linear differential equations.
Then by the initial condition \eqref{initial} and the uniqueness of solutions, we conclude, in particular, that
\begin{equation}
\mu^{V}=\eta^{K}=\zeta^{c}=0,
\quad V>q-r,~ K>n, ~ c>r.
\end{equation}
Hence \eqref{adapted1}, \eqref{adapted1'} imply
\begin{align}
\widetilde Z_\alpha&=\lambda_\alpha^{~b}Z_b+\mu_\a^{~v}Z'_v+\eta_\alpha^{~k}X_k+\zeta_\alpha^{~\beta}Y_\beta,\Label{Z-al1}\\
\widetilde Z'_u&=\lambda_u^{~b}Z_b+\mu_u^{~v}Z'_v+\eta_u^{~k}X_k+\zeta_u^{~\beta}Y_\beta.\Label{Z-al2}
\end{align}

Now setting
\begin{align}
\3Z_{\a}:&= \2Z_{\a} - \sum_{b>r} \l_{\a}^{~b} Z_{b},\\
\3Z'_u:&=\2Z'_u-\sum_{b>r} \l_{u}^{~b} Z_{b},\\
\3Z'_U:&=\2Z'_U-\sum_{b>r} \l_{U}^{~b} Z_{b},
\end{align}
we still have
\begin{equation}
\sp \{ \3Z_{\a},  \2Z_{r+1},\ldots, \2Z_{r'}, \3Z'_U\} =  
	\sp \{ \2Z_{a}, \2Z'_U \},
\end{equation}
whereas \eqref{Z-al1}, \eqref{Z-al2} and \eqref{adapted1'} become
\begin{align}
\3Z_\alpha &=\lambda_\alpha^{~\beta}Z_\beta+\mu_\a^{~v}Z'_v+\eta_\alpha^{~k}X_k+\zeta_\alpha^{~\beta}Y_\beta,\\
\3Z'_u &=\lambda_u^{~\beta}Z_\beta+\mu_u^{~v}Z'_v+\eta_u^{~k}X_k+\zeta_u^{~\beta}Y_\beta\\
\3Z'_U &=\lambda_U^{~\beta}Z_\beta+\mu_U^{~V}Z'_V+\eta_U^{~K}X_K+\zeta_U^{~b}Y_b, \quad U>q-r,
\end{align}
implying
$$\sp \{\3Z_{\a}, \3Z'_u\} \subset 
	\sp  \{ Z_\a, Z'_u,X_k, Y_\b\}=:V_0,$$
and since $\2Z'_U$ is in the span of $Z_a, Z'_V, X_J $ with
$a>r, ~V>q-r, ~ J>n$,
$$
\3Z'_U = \sum_{V>q-r} \mu_U^{~V}Z'_V
	+ \sum_{K>n}\eta_U^{~K}X_K, \quad U>q-r,
$$
implying
$$\sp \{ \3Z'_U\} \subset 
	\sp  \{ Z'_U,X_J : U>q-r, J>n\}=V_2.$$
Then together with
\eqref{dZa} we conclude that for $x\in M$,
\begin{align*}
f(x)=\sp\{ \2Z_{a}, \2Z'_U\}
=&~\sp\{\3Z_{\a}, \3Z'_u\} \oplus
	\sp\{\2Z_{r+1},\ldots,\2Z_{r'}\}
 	\oplus \sp\{\3Z'_{q-r+1},\ldots, \3Z'_{q'-r'} \} 
	\\
=&~\sp\{\3Z_{\a}, \3Z_u\} \oplus
	\sp\{Z_{r+1},\ldots,Z_{r'}\}
 	\oplus \sp\{\3Z'_{q-r+1},\ldots, \3Z'_{q'-r'} \} \\
%=&~V_1\oplus S_2\oplus S_3
\in& ~Gr(V_0,q) \oplus V_1 \oplus  Gr(V_2, (q'-r')-(q-r)).
\end{align*}
\end{proof}

\section{Classification of CR maps between boundary components}\Label{proofs}
\bpf[Proof of Theorem~\ref{main}]
Let $V_0$ be given by Proposition~\ref{flatness}.
After a linear change of coordinates in $\C^{p'+q'}$
preserving the basic form (that corresponds to an automorphism of $D_{p',q'}$),
we may assume that 
$V_0=\C^{p+q}\times\{0\}$
and hence the $Gr(V_0,q)$-component of $f$ in \eqref{6.1}
defines a local CR diffeomorphism of $S_{p,q,r}$.
Then by a theorem of Kaup-Zaitsev \cite[Theorem~4.5]{KZ06},
the $Gr(V_0,q)$-component of $f$ is a restriction of a global CR-automorphism of $S_{p,q,r}$.
Furthermore, by \cite[Theorem~8.5]{KZ00},
the $Gr(V_0,q)$-component of $f$ extends to a biholomorphic automorphism
of the bounded symmetric domain $D_{p,q}$.
Hence, composing $f$ with a suitable automorphism of $D_{p',q'}$
we can put $f$ in the form \eqref{f}.
Since $f(x)\in S_{p',q',r'}$, it follows from the description of $S_{p',q',r'}$
that in the notation of \eqref{f} we must have \eqref{h-cond}.
Vice versa, any $f$ of the form \eqref{f} with $h$ satisfying \eqref{h-cond}
defines a CR map between pieces of $S_{p,q,r}$ and $S_{p',q',r'}$
satisfying the assumptions of Theorem~\ref{main}. The proof is complete.
\epf

\bpf[Proof of Theorem~\ref{cor0}]
Let $f$ be as in the corollary.
Consider the restriction $\2f$ of $f$ to the hypersurface boundary component $S_{p,q,1}$.
Then $\2f$ restricts to a CR map between open pieces
of $S_{p,q,1}$ and $S_{p',q',r'}$ for some $1\le r'\le q'$.
Since $S_{p,q,1}$ is a real hypersurface,
the transversality assumption 
$df(\xi)\in T'\setminus T'^c$ for $\xi \in T\setminus T^c$
of Theorem~\ref{main}
is satisfied.
Indeed, otherwise we would have $df(T)\subset T'^c$
and the Levi form identity \eqref{levi-id} would imply
that $df(T)$ is contained in the Levi null-space of $S_{p',q',r'}$,
which, in view of positivity, coincides with the kernel.
The latter is an integrable distribution whose orbits
are complex submanifolds of $S_{p',q',r'}$.
Then $f$ would send any curve in $S_{p,q,1}$ into 
one of these complex submanifolds 
(see \cite{BER99} for details).
Hence
it would follow that $f$ sends an open piece
of $S_{p,q,1}$ into a complex submanifold of $S_{p',q',r'}$,
which would contradict the assumptions of corollary.

Next, since $r'\ge r=1$, the assumptions \eqref{cor-ineq}
imply \eqref{main-ineq}.
Now by Theorem~\ref{main}, we can assume
that $f$ is of the form \eqref{f}.
Furthermore, the assumption $f(U\cap D_{p,q})\not\subset \d D_{p',q'}$
implies that the block $I_{r'-r}$ in \eqref{f} must be trivial, and hence
$f$ is of the desired form.
\epf

%\section{Proof of Corollary~\ref{cor}}

\bigskip
{\bf Acknowledgement.}
The authors are grateful to Wilhelm Kaup for careful reading and helpful remarks.


\begin{thebibliography}{BER96b}

\bibitem[A74]{A74}
{\bf Alexander, H.} --- 
Holomorphic mappings from the ball and polydisc, {\em Math. Ann.} {\bf 209} (1974), 249--256.

\bibitem[BEH08]{BEH08}
{\bf Baouendi, M. S.; Ebenfelt, P.; Huang, X.} --- Super-rigidity for CR embeddings of real hypersurfaces into hyperquadrics. {\em Adv. Math.} {\bf 219} (2008), no. 5, 1427--1445.

\bibitem[BEH09]{BEH09}
{\bf Baouendi, M. S.; Ebenfelt, P.; Huang, X.} ---
Holomorphic Mappings between Hyperquadrics with Small Signature Difference.
{\em American J.  Math.} {\bf 133}  (2011), no. 6, 
1633--1661.

\bibitem[BER99]{BER99} {\bf Baouendi, M.S.; Ebenfelt,~P.; Rothschild,~L.P.} --- {\em Real Submanifolds in Complex Space and Their Mappings}. Princeton Math. Series {\bf 47}, Princeton Univ. Press, 1999.

\bibitem[BH05]{BH05}
{\bf Baouendi, M. S.; Huang, X.} --- Super-rigidity for holomorphic mappings between hyperquadrics with positive signature. {\em J. Differential Geom.} {\bf 69} (2005), no. 2, 379--398.

\bibitem[Bo47]{Bo47} {\bf Bochner, S.} Curvature in Hermitian metric. {\em Bull. Amer. Math. Soc.} {\bf 53} (1947), 179--195. 

\bibitem[Ca53]{Ca53}
{\bf Calabi, E.} --- Isometric imbedding of complex manifolds. {\em Ann. of Math.} (2) {\bf 58} (1953), 1--23.

\bibitem[CS83]{CS83} {\bf Cima, J.; Suffridge, T. J.} ---
A reflection principle with applications to proper  holomorphic
mappings. {\em Math. Ann.} {\bf 265} (1983), 489--500.

\bibitem[CS90]{CS90} {\bf Cima, J.; Suffridge, T. J.}  ---
Boundar behavior of rational proper maps. {\em Duke Math.} {\bf
60} (1990), 135--138.

\bibitem[E13]{E13} 
{\bf Ebenfelt, P.} --- 
Partial rigidity of degenerate CR embeddings into spheres. 
{\em Adv. Math.} {\bf 239} (2013), 72--96.

\bibitem[EHZ04]{EHZ04} {\bf Ebenfelt, P.; Huang, X.; Zaitsev, D.} --- Rigidity of CR-immersions into spheres. {\em Comm. in Analysis and Geometry} {\bf 12} (2004), no. 3, 631-670.

\bibitem[EHZ05]{EHZ05} {\bf Ebenfelt, P.; Huang, X.; Zaitsev, D.} --- The equivalence problem and rigidity for hypersurfaces embedded into hyperquadrics. {\em Amer. J. Math.} {\bf 127} (2005), no. 1, 169-191.

\bibitem[EM12]{EM12} {\bf Ebenfelt, P.; Minor, A.} ---
On CR embeddings of strictly pseudoconvex hypersurfaces into spheres in low dimensions.
{\sf http://arxiv.org/abs/1208.0947}

\bibitem[ES10]{ES10} {\bf Ebenfelt, P.; Shroff, R.} --- Partial Rigidity of CR Embeddings of Real Hypersurfaces into Hyperquadrics with Small Signature Difference.
{\sf http://arxiv.org/abs/1011.1034}

\bibitem [Fa86]{Fa86} {\bf Faran V, J.J.} --- 
On the linearity of proper maps
between balls in the lower codimensional case. {\em J.
Differential Geom.} {\bf 24}, (1986), 15--17.

\bibitem[Fo86a]{F86a}
{\bf Forstneri\v c, F.} ---
Embedding strictly pseudoconvex domains into balls. 
{\em Trans. Amer. Math. Soc.} {\bf 295} (1986), no. 1, 347--368. 

\bibitem [Fo86b]{F86b} {\bf Forstneri\v c, F.} ---
Proper holomorphic maps between balls. {\em Duke Math. J.} {\bf
53} (1986), 427--440.

\bibitem [Fo89]{F89} {\bf Forstneri\v c, F.} ---
Extending proper holomorphic mappings of positive codimension.
{\em Invent. Math.} {\bf 95} (1989), 31--61.

\bibitem[G87]{G}
{\bf Globevnik, J.} ---
Boundary interpolation by proper holomorphic maps. 
{\em Math. Z.} {\bf 194} (1987), no. 3, 365--373. 

\bibitem[HS83]{HS}
{\bf Hakim, M.; Sibony, N.} --- Fonctions holomorphes born\'ees et limites tangentielles. 
{\em Duke Math. J.}  {\bf 50} (1983), no. 1, 133--141.


\bibitem [H99]{H99} {\bf Huang, X.} ---
On a linearity problem for proper holomorphic maps between balls
in complex spaces of different dimensions. {\em J. Differential Geom.} {\bf 51} (1999), 13--33.

\bibitem [HJ01]{HJ01} {\bf Huang, X.; Ji, S.}  ---
Mapping $B^{n}$ into $B^{2n-1}$. {\em Invent. Math.}  {\bf 145} (2) (2001), 219--250.

\bibitem [H03]{H03}
{\bf Huang, X.} --- On a semi-rigidity property for holomorphic maps. {\em Asian J. Math.} {\bf 7} (2003), no. 4, 463--492. 

\bibitem [HJX06]{HJX06} 
{\bf Huang, X.; Ji, S.; Xu, D.} --- A new gap phenomenon for proper holomorphic mappings from $B^{n}$ into $B^{N}$. {\em Math. Res. Lett.} {\bf 13} (4) (2006), 515--529.

\bibitem [HJX07]{HJX07} 
{\bf Huang, X.; Ji, S.; Xu, D.} --- 
On some rigidity problems in Cauchy-Riemann analysis.
{\em Proceedings of the International Conference on Complex Geometry and Related Fields}, 
89 -- 107, AMS/IP Stud. Adv. Math., 39, Amer. Math. Soc., Providence, RI, 2007. 


\bibitem[KaZ00]{KZ00} {\bf Kaup,~W.; Zaitsev,~D.} ---
On symmetric Cauchy-Riemann manifolds.
{\em Adv. Math.}  {\bf 149} (2000), no. 2, 145--181.


\bibitem[KaZ03]{KZ03} {\bf Kaup, W.; Zaitsev, D.} --- On the CR-structure of compact group orbits associated with bounded symmetric domains. 
{\em Invent. Math.} {\bf 153} (2003), no. 1, 45--104.


\bibitem[KaZ06]{KZ06} {\bf Kaup, W.; Zaitsev, D.} --- On local CR-transformations of Levi-degenerate
group orbits in compact Hermitian symmetric spaces.
{\em J. Eur. Math. Soc.} {\bf 8} (2006), 465--490.

\bibitem[KiZ12]{KZ} {\bf Kim, S.Y.; Zaitsev, D.} --- Rigidity of CR maps between Shilov boundaries of bounded symmetric domains. 
{\em Invent. Math.} {\bf 193} (2013), no. 2, 409--437.

\bibitem[L85]{Lo}
{\bf L\o w, E.}
Embeddings and proper holomorphic maps of strictly pseudoconvex domains into polydiscs and balls. 
{\em Math. Z.} {\bf 190} (1985), no. 3, 401--410. 

\bibitem[M89]{M89}
{\bf Mok, N.} --- Metric Rigidity Theorems on Hermitian Locally Symmetric Spaces. 
Series in Pure Math., Vol. 6, World Scientific, Singapore, 1989.

\bibitem[M08]{M08}
{\bf Mok, N.} --- Nonexistence of proper holomorphic maps between certain classical bounded symmetric domains. {\em Chin. Ann. Math. Ser.} B {\bf 29} (2008), no. 2, 135--146. 

\bibitem[M11]{M11}
{\bf Mok, N.} --- Geometry of holomorphic isometries and related maps between bounded
domains, in Geometry and Analysis, Vol. II, ALM 18, Higher Education Press and
International Press, Beijing-Boston 2011, 225--270.

\bibitem[MN12]{MN12}
{\bf Mok, N.; Ng, S.-C.} --- Germs of measure-preserving holomorphic maps from bounded symmetric domains to their Cartesian products. {\em J. Reine Angew. Math.} {\bf 669} (2012), 47--73.

\bibitem[MNT10]{MNT10}
{\bf Mok, N.; Ng, S.-C.; Tu, Z.} --- Factorization of proper holomorphic maps on irreducible bounded symmetric domains of rank $\ge 2$. {\em Sci. Chi. Math.} {\bf 53}, No. 3, (2010), 813--826.

\bibitem[Ng12]{Ng12}
{\bf Ng, S.-C.} --- Cycle spaces of flag domains on Grassmannians and rigidity of holomorphic mappings. 
{\em Math. Res. Lett.} {\bf 19}, No. 6, (2012), 1219--1236.

\bibitem[Ng13a]{Ng13a}
{\bf Ng, S.-C.} --- On proper holomorphic mappings among irreducible
bounded symmetric domains of rank at least 2.
{\em Proc. Amer. Math. Soc.} (to appear).

\bibitem[Ng13b]{Ng13b}
{\bf Ng, S.-C.} --- Holomorphic double fibration and the mapping problems of classical domains.
{\em Int. Math. Res. Not.} (to appear).

\bibitem[Ng13c]{Ng13c}
{\bf Ng, S.-C.} --- Proper holomorphic mappings on $SU(p,q)$-type flag domains on projective spaces. 

\bibitem[P07]{P07} {\bf Poincar\'e, H.} --- 
Les fonctions analytiques de deux variables et la repr\'esentation conforme, {\em Rend. Circ. Mat. Palermo}
(2) {\bf 23} (1907), 185--220.

\bibitem[S80]{S80} {\bf Siu, Y.-T.}  ---
The complex-analyticity of harmonic maps and the strong rigidity of compact K\"ahler manifolds.
{\em Ann. of Math. (2)}  {\bf 112} (1980), no. 1, 73--111.

\bibitem[S81]{S81} {\bf Siu, Y.-T.}  ---
Strong rigidity of compact quotients of exceptional bounded symmetric domains, {\em Duke Math. J.} {\bf 48} (1981), 857--871.

\bibitem[St96]{St}
{\bf Stens\o nes, B.} ---
Proper maps which are Lipschitz $\alpha$ up to the boundary. 
{\em J. Geom. Anal.} {\bf 6} (1996), no. 2, 317--339. 

\bibitem[Ts93]{T93} {\bf Tsai, I-H.} --- 
Rigidity of proper holomorphic maps between symmetric domains.
{\em J. Differential Geom.} {\bf 37} (1993), no. 1, 123--160.


\bibitem[Tu02b]{Tu02b} {\bf Tu, Z.} ---
Rigidity of proper holomorphic mappings between nonequidimensional bounded symmetric domains. {\em Math. Z.}  {\bf 240} (2002), no. 1, 13--35. 

\bibitem[Tu02a]{Tu02a} {\bf Tu, Z.} --- Rigidity of proper holomorphic mappings between equidimensional bounded symmetric domains. 
{\em Proc. Amer. Math. Soc.} {\bf 130} (2002), no. 4, 1035--1042.


\bibitem[W79]{W79} {\bf Webster, S. M.} --- The rigidity of C-R hypersurfaces
in a sphere. {\em Indiana Univ. Math. J.} {\bf 28} (1979), 405--416.
\end{thebibliography}
\end{document}